\documentclass[reqno]{amsart}
\usepackage{amsmath}
\usepackage{amsthm}
\usepackage{amssymb}
\usepackage{mathrsfs} 
\usepackage{bm}
\usepackage{courier}
\usepackage{xcolor}
\usepackage{hyperref}
\hypersetup{
    colorlinks,
    linkcolor={blue!50!black},
    citecolor={blue!50!black},
    urlcolor={blue!80!black}
}

\usepackage{comment}
\usepackage{tikz}
\usetikzlibrary{arrows.meta}
\usetikzlibrary{shapes.misc}

\newcommand{\strsub}{\prec}

\renewcommand{\emptyset}{\varnothing}

\newcommand{\C}{{\mathcal C}}

\newcommand{\B}{{\mathcal B}}

\newcommand{\R}{{\mathcal R}}
\renewcommand{\O}{{\mathcal O}}
\newcommand{\I}{{\mathcal I}}

\newcommand{\restrict}{\upharpoonright}

\newcommand{\<}{\langle}
\renewcommand{\>}{\rangle}

\newcommand{\st}{\mid}

\newcommand{\ot}{\mathop{\rm ot}\nolimits}
\newcommand{\cf}{\mathop{\rm cf}}

\newcommand{\NS}{{\mathop{\rm NS}}}

\newcommand{\NSS}{{\mathop{\rm NSS}}}

\renewcommand{\and}{\mathop{\&}}

\newtheorem{theorem}{Theorem}[section]
\newtheorem{lemma}[theorem]{Lemma}
\newtheorem{corollary}[theorem]{Corollary}
\newtheorem{proposition}[theorem]{Proposition}
\newtheorem{claim}[theorem]{Claim}

\newtheorem{fact}[theorem]{Fact}

\theoremstyle{definition}

\newtheorem{question}[theorem]{Question}
\newtheorem{remark}[theorem]{Remark}

\newtheorem{definition}[theorem]{Definition}

\date{\today}

\begin{document}

\title{Two-cardinal derived topologies, indescribability and Ramseyness}

\author[Brent Cody]{Brent Cody}
\address[Brent Cody]{ 
Virginia Commonwealth University,
Department of Mathematics and Applied Mathematics,
1015 Floyd Avenue, PO Box 842014, Richmond, Virginia 23284, United States
} 
\email[B. ~Cody]{bmcody@vcu.edu} 
\urladdr{http://www.people.vcu.edu/~bmcody/}

\author[Chris Lambie-Hanson]{Chris Lambie-Hanson}
\address[Chris Lambie-Hanson]{ 
Institute of Mathematics,
Czech Academy of Sciences,
\v{Z}itn\'{a} 25,   115 67   Praha 1,   Czech Republic
} 
\email[C. ~Lambie-Hanson]{lambiehanson@math.cas.cz} 
\urladdr{https://users.math.cas.cz/~lambiehanson/}

\author[Jing Zhang]{Jing Zhang}
\address[Jing Zhang]{ 
Department of Mathematics, University of Toronto
Bahen Centre, Room 6290
40 St. George St., Toronto, ON, M5S 2E4
} 
\email[J. ~Zhang]{jingzhan@alumni.cmu.edu} 
\urladdr{https://jingjzzhang.github.io/}

\thanks{The second author was supported by GA\v{C}R project 23-04683S 
and the Czech Academy of Sciences (RVO 67985840). A portion of this work was carried out while all three authors 
were participating in the Thematic Program on Set Theoretic Methods in 
Algebra, Dynamics and Geometry at the Fields Institute in spring of 2023. 
We thank the Fields Institute for their support and hospitality.}

\begin{abstract}
We introduce a natural two-cardinal version of Bagaria's sequence of derived topologies on ordinals. We prove that for our sequence of two-cardinal derived topologies, limit points of sets can be characterized in terms of a new iterated form of pairwise simultaneous reflection of certain kinds of stationary sets, the first few instances of which are often equivalent to notions related to strong stationarity, which has been studied previously in the context of strongly normal ideals \cite{MR1074449}. The non-discreteness of these two-cardinal derived topologies can be obtained from certain two-cardinal indescribability hypotheses, which follow from local instances of supercompactness. Additionally, we answer several questions posed by the first author, Peter Holy and Philip White on the relationship between Ramseyness and indescribability in both the cardinal context and in the two-cardinal context.
\end{abstract}



\maketitle



%

\section{Introduction}

The \emph{derived set} of a subset $A$ of a topological space $(X,\tau)$ is the collection $d(A)$ of all limit points of $A$ in the space. We refer to the function $d$ as the Cantor derivative of the space $(X,\tau)$. Recently, Bagaria showed \cite{MR3894041} that the \emph{derived topologies on ordinals}, whose definition we review now, are closely related to certain widely studied stationary reflection properties and large cardinal notions. Suppose $\delta$ is an ordinal and $\tau_0$ is the order topology on $\delta$. That is, $\tau_0$ is the topology on $\delta$ generated by $B_0=\{\{0\}\}\cup\{(\alpha,\beta)\st\alpha<\beta<\delta\}$. For a set $A\subseteq\delta$, it easily follows that the collection $d_0(A)$ of all limit points of $A$ in the space $(\delta,\tau_0)$, is equal to $\{\alpha<\delta\st \text{$A$ is unbounded in $\alpha$}\}$. Beginning with the interval topology on $\delta$ and declaring more and more derived sets to be open, Bagaria \cite{MR3894041} introduced the sequence of derived topologies $\<\tau_\xi\st\xi<\delta\>$ on $\delta$. For example, $\tau_1$ is the topology on $\delta$ generated by $B_1=B_0\cup\{d_0(A)\st A\subseteq\delta\}$, and $\tau_2$ is the topology on $\delta$ generated by $B_2=B_1\cup\{d_1(A)\st A\subseteq\delta\}$ where $d_1$ is the Cantor derivative of the space $(\delta,\tau_1)$. Bagaria showed that limit points of sets in the spaces $(\delta,\tau_\xi)$, for $\xi\in\{1,2\}$, can be characterized as follows. For $A\subseteq\delta$ and $\alpha<\delta$: $\alpha$ is a limit point of $A$ in $(\delta,\tau_1)$ if and only if $A$ is stationary in $\alpha$, and $\alpha$ is a limit point of $A$ in $(\delta,\tau_2)$ if and only if whenever $S$ and $T$ are stationary subsets of $\alpha$ there is a $\beta\in A$ such that $S\cap\beta$ and $T\cap\beta$ are stationary subsets of $\beta$. Furthermore, Bagaria proved that limit points of sets in the spaces $(\delta,\tau_\xi)$ for $\xi>2$ can be characterized in terms of an iterated form of pairwise simultaneous stationary reflection called $\xi$-s-stationarity.



In this article we address the following natural question: is there some analogue of the sequence of derived topologies on an ordinal in the two-cardinal setting? Specifically, suppose $\kappa$ is a cardinal and $X$ is a set of ordinals with $\kappa\subseteq X$. Is there a topology $\tau_\xi$ on $P_\kappa X$ such that, for all $A \subseteq 
P_\kappa X$, the limit points of $A$ in the space $(P_\kappa X,\tau_\xi)$ are precisely the points $x\in P_\kappa X$ such that the set $A$ satisfies: 
\begin{itemize}
\item some \emph{unboundedness} condition at $x$? 
\item some \emph{stationarity} condition at $x$? 
\item some \emph{pairwise simultaneous stationary reflection-like} condition at $x$?
\end{itemize}


Recall that for $x,y\in P_\kappa X$ we say that $x$ is a \emph{strong subset} of $y$ and write $x\strsub y$ if $x\subseteq y$ and $|x|<|y\cap\kappa|$. Let us note that the ordering $\strsub$, and its variants, are used in the context of supercompact Prikry forcings \cite{MR2768695}. In Section \ref{section_generalizing_the_order_topology}, we show that the ordering $\strsub$ induces a natural topology $\tau_0$ on $P_\kappa X$ analogous to the order topology on an ordinal $\delta$. Furthermore, beginning with $\tau_0$ and following the constructions of \cite{MR3894041}, in Section \ref{section_derived_topologies} we define a sequence of \emph{derived topologies} $\<\tau_\xi\st\xi<\kappa\>$ on $P_\kappa X$. Let us note that after submitting the current article, the authors learned that Catalina Torres, working under the supervision of Joan Bagaria, simultaneously and independently defined a sequence of two-cardinal derived topologies and obtained results similar to those in Sections \ref{section_derived_topologies} - \ref{section_variations} involving the relationship between various two-cardinal notions of $\xi$-s-stationarity and two-cardinal derived topologies. 

We show (see Propositions \ref{proposition_strong_vs_1_s_stationarity} and \ref{proposition_limit_points_tau_1}) that in the space $(P_\kappa X,\tau_1)$, for $x\in P_\kappa X$ with $x\cap\kappa$ an inaccessible cardinal, $x$ is a limit point of a set $A\subseteq P_\kappa X$ if and only if $A$ is \emph{strongly stationary in $P_{x\cap\kappa}x$} (see Section \ref{section_strong_stationarity} for the definition of strongly stationary set). Let us note that although the notion of strong stationarity is distinct from the widely popular notion of two-cardinal stationarity introduced by Jech \cite{MR0325397} (see \cite[Lemma 2.2]{MR4082998}), it has previously been studied by several authors \cite{MR1074449, MR4082998, MR0954259, MR2204569, MR2440418}. The analogy with the case of derived topologies on ordinals continues: in the space $(P_\kappa X,\tau_2)$, when $x\in P_\kappa X$ is such that $x\cap\kappa<\kappa$ and $P_{x\cap\kappa}x$ satisfies a two-cardinal version of $\Pi^1_1$-indescribability, $x$ is a limit point of a set $A\subseteq P_\kappa X$ if and only if for every pair $S,T$ of strongly stationary subsets of $P_{\kappa\cap x}x$ there is a $y\strsub x$ in $A$ with $y\cap\kappa<x\cap\kappa$ such that $S$ and $T$ are both strongly stationary in $P_{y\cap\kappa}y$ (see Proposition \ref{proposition_indescribable_reflection}). Additionally, using a different method, we show (see Corollary \ref{corollary_stationary}) that if $\kappa$ is weakly inaccessible and $X$ is a set of ordinals with $\kappa\subseteq X$, then there is a topology on $P_\kappa X$ such that for $A\subseteq P_\kappa X$, $x\in P_\kappa X$ is a limit point of $A$ if and only if $\kappa_x$ is weakly inaccessible and $A$ is stationary in $P_{\kappa_x}x$ in the sense of Jech \cite{MR0325397}. 

In order to prove the characterizations of limit points of sets in the spaces $(P_\kappa X,\tau_\xi)$ (Theorem \ref{theorem_induction}(1)), we introduce new iterated forms of two-cardinal stationarity and two-cardinal pairwise simultaneous stationary reflection, which we refer to as $\xi$-strong stationarity and $\xi$-s-strong stationarity (see Definition \ref{definition_xi_s_stationary}). Let us note that the notions of $\xi$-strong stationarity and $\xi$-s-strong stationarity introduced here are natural generalizations of notions previously studied in the cardinal context by Bagaria, Magidor and Sakai \cite{MR3416912}, Bagaria \cite{MR3894041} and by Brickhill and Welch \cite{MR4583072}, as well as those previously studied in the two-cardinal context by Sakai \cite{Sakai}, by Torres \cite{torres_ms_thesis}, as well as by Benhamou and the third author \cite{BZ}.

We establish some basic properties of the ideals associated to $\xi$-strong stationarity and $\xi$-s-strong stationarity and introduce notions of $\xi$-weak club and $\xi$-s-weak club which provide natural filter bases for the corresponding ideals (see Corollary \ref{corollary_xi_plus_1_s_filter_base}). The consistency of the non-discreteness of the derived topologies $\tau_\xi$ on $P_\kappa X$ is obtained using various two-cardinal indescribability hypotheses, all of which follow from appropriate local instances of supercompactness (see Section \ref{section_indescribability}). We also show that 
by restricting our attention to a certain natural club subset of 
$P_\kappa X$, some questions about the resulting spaces, such as questions regarding when particular subbases are in fact bases, become 
more tractable (see Section \ref{section_variation}).

Additionally, in Section \ref{section_ramsey}, we answer several questions asked by the first author and Peter Holy \cite{MR4594301} and the first author and Philip White \cite{CodyWhite} concerning the relationship between Ramseyness and indescribability. For example, answering \cite[Question 10.9]{MR4594301} in the affirmative, we show that the existence of a $2$-Ramsey cardinal is strictly stronger in consistency strength than the existence of a $1$-$\Pi^1_1$-Ramsey cardinal. In other words, the existence of an uncountable cardinal $\kappa$ such that for every regressive function $f:[\kappa]^{<\omega}\to\kappa$ there is a set $H\subseteq\kappa$ which is positive for the Ramsey ideal homogeneous for $f$, is strictly stronger in consistency strength than the existence of a cardinal $\kappa$ such that for every regressive function $f:[\kappa]^{<\omega}\to\kappa$ there is a set $H\subseteq\kappa$ that is positive for the $\Pi^1_1$-indescribability ideal and homogeneous for $f$.



\section{Strong stationarity and weak clubs}\label{section_strong_stationarity}

Suppose $\kappa$ is a cardinal and $X$ is a set of ordinals with $\kappa\subseteq X$. Given $x \in P_\kappa X$, we denote $|x \cap \kappa|$ 
by $\kappa_x$. We define an ordering $\strsub $ on $P_\kappa X$ by letting $x\strsub y$ if and only if $x\subseteq y$ and $|x|< \kappa_y$. An ideal $I$ on $P_\kappa X$ is \emph{strongly normal} if whenever $S\in I^+$ and $f:S\to P_\kappa X$ is such that $f(x)\strsub x$ for all $x\in S$, then there is some $T\in P(S)\cap I^+$ such that $f\restrict T$ is constant. It is easy to see that an ideal $I$ is strongly normal if and only if the dual filter $I^*$ is closed under $\strsub$-diagonal intersections in the following sense: whenever $A_x\in I^*$ for all $x\in P_\kappa X$, the $\strsub$-diagonal intersection 
\[\bigtriangleup_\strsub\{A_x\st x\in P_\kappa X\}=\{y\in P_\kappa X\st y\in\bigcap_{x\strsub y}A_x\}\]
is in $I^*$. Carr, Levinski and Pelletier \cite{MR1074449} showed that there is a strongly normal ideal on $P_\kappa X$ if and only if $\kappa$ is a Mahlo cardinal or $\kappa=\mu^+$ for some cardinal $\mu$ with $\mu^{< \mu}=\mu$. Furthermore, they proved that when a strongly normal ideal exists on $P_\kappa X$, the minimal such ideal is that consisting of the non-strongly stationary subsets of $P_\kappa X$, which are defined as follows. Given a function $f:P_\kappa X\to P_\kappa X$ we let
\[B_f=\{x\in P_\kappa X\st x\cap\kappa\neq\emptyset\land f[P_{\kappa_x}x]\subseteq P(x)\}.\]
A set $S\subseteq P_\kappa X$ is \emph{strongly stationary in $P_\kappa X$} if for all $f:P_\kappa X\to P_\kappa X$ we have $S\cap B_f\neq\emptyset$. 
The non-strongly stationary ideal on $P_\kappa X$ is the collection
\[\NSS_{\kappa,X}=\{X\subseteq P_\kappa X\st \text{$X$ is not strongly stationary}\}.\]
Thus, when $\kappa$ is Mahlo or $\kappa=\mu^+$ where $\mu^{< \mu}=\mu$, the ideal $\NSS_{\kappa,X}$ is the minimal strongly normal ideal on $P_\kappa X$.

When $\kappa$ is Mahlo, we can identify a filter base for the filter dual to $\NSS_{\kappa,X}$ consisting of sets which are, in a sense, cofinal in $P_\kappa X$ and satisfy a certain natural closure property. We say that a set $C\subseteq P_\kappa X$ is \emph{$\strsub $-cofinal in $P_\kappa X$} if for all $x\in P_\kappa X$ there is a $y\in C$ such that $x\strsub y$. A set $C\subseteq P_\kappa X$ is said to be a \emph{weak club in $P_\kappa X$} if $C$ is $\strsub $-cofinal in $P_\kappa X$ and \emph{$\strsub$-closed in $P_\kappa X$}, meaning that for all $x\in P_\kappa X$, if $C$ is $\strsub $-cofinal in $P_{\kappa_x}x$ then $x\in C$. Given a function $f:P_\kappa X\to P_\kappa X$ let 
\[C_f=\{x\in P_\kappa X\st x\cap\kappa\neq\emptyset\land f[P_{\kappa_x}x]\subseteq P_{\kappa_x}x\}.\]

\begin{fact}
If $\kappa$ is a Mahlo cardinal, the sets
\[\mathcal{C}_0=\{B_f\st f:P_\kappa X\to P_\kappa X\},\]
\[\mathcal{C}_1=\{C_f\st f:P_\kappa X\to P_\kappa X\}\]
and
\[\mathcal{C}_2=\{C\subseteq P_\kappa X\st \text{$C$ is weak club in $P_\kappa X$}\}\]
generate the same filter on $P_\kappa X$, namely, the filter $\NSS_{\kappa,X}^*$ dual to the ideal $\NSS_{\kappa,X}$.
\end{fact}

\begin{proof}

By definition, the filter on $P_\kappa X$ generated by $\C_0$ is $\NSS_{\kappa,X}^*$. 

Let us show that the filter generated by $\C_1$ equals that generated by $\C_2$. Suppose $C\in \C_2$ is weak club in $P_\kappa X$. Define $f:P_\kappa X\to P_\kappa X$ by letting $f(x)$ be some member of $C$ with $x\strsub f(x)$. Then $C_f\subseteq C$ because if $x\in C_f$ then $C$ is $\strsub $-cofinal in $P_{\kappa_x}x$ and hence $x\in C$. 

For the other direction, we fix a function $g:P_\kappa X\to P_\kappa X$ and show that $C_g$ is weak club in $P_\kappa X$. First let us check that $C_g$ is $\strsub$-cofinal in $P_\kappa X$. Fix $x\in P_\kappa X$. We define an increasing chain $\<x_\eta\st\eta<\kappa\>$ in $P_\kappa X$ as follows. Let $x_0=x$. Given $x_\eta$ we choose $x_{\eta+1}\in P_\kappa X$ with $\kappa_{x_{\eta+1}}=x_{\eta+1}\cap\kappa$ and $\bigcup f[P_{\kappa_{x_\eta}}x_\eta]\strsub x_{\eta+1}$. When $\eta<\kappa$ is a limit ordinal we let $x_\eta=\bigcup_{\alpha<\eta}x_\alpha$. Then $\<\kappa_{x_\eta}\st\eta<\kappa\>$ is a strictly increasing sequence in $\kappa$ and the set
\[C=\{\eta<\kappa\st(\forall\zeta<\eta)\kappa_{x_\zeta}<\eta\}=\{\eta<\kappa\st\kappa_{x_\eta}=\eta\}\]
is club in $\kappa$. Since $\kappa$ is Mahlo, we can fix some regular $\kappa_{x_\eta}=\eta\in C$. Clearly $x\strsub x_\eta$. Let us show that $x_\eta\in C_g$. Suppose $a\in P_{\kappa_{x_\eta}}x_\eta$. Since $\kappa_{x_\eta}=\eta$ is regular, $|a|<\kappa_{x_\eta}$ implies that $a\in P_{\kappa_{x_\zeta}}x_\zeta$ for some $\zeta<\eta$, and therefore 
\[g(a)\subseteq \bigcup g[P_{\kappa_{x_\zeta}}x_\zeta]\strsub x_{\zeta+1},\]
which implies $g(a)\in P_{\kappa_{x_\eta}}x_\eta$ and hence $x_\eta\in C_g$. Since $x\strsub x_\eta$, it follows that $C_g$ is $\strsub$-cofinal.

Now we verify that $C_g$ is $\strsub$-closed in $P_\kappa X$. Suppose $C_g\cap P_{\kappa_x}x$ is $\strsub $-cofinal in $P_{\kappa_x}x$. We must show that $x\in C_g$. Suppose $y\in P_{\kappa_x}x$. Then there is some $z\in C_g$ with $y\strsub z\strsub x$. Thus $g(y)\strsub z\strsub x$ and hence $x\in C_g$. 


Now let us verify that the filter generated by $\C_0$ equals that generated by $\C_1$. For any function $f:P_\kappa X\to P_\kappa X$ we have $C_f\subseteq B_f$, so the filter generated by $\C_0$ is contained in the filter generated by $\C_1$. Let us fix a function $g:P_\kappa X\to P_\kappa X$. We must show that there is a function $h:P_\kappa X\to P_\kappa X$ such that $B_h\subseteq C_g$. Define $h:P_\kappa X\to P_\kappa X$ by letting $h(x)$ be some member of $C_g$ with $g(x)\strsub h(x)$, for all $x\in P_\kappa X$. Suppose $x\in B_h$. To show $x\in C_g$, suppose $y\strsub x$. Then it follows that $g(y)\strsub h(y)\subseteq x$, which implies $g(y)\strsub x$ and thus $x\in C_g$. Therefore $B_h\subseteq C_g$ and hence the filter generated by $\C_0$ equals the filter generated by $\C_1$.
\end{proof}


We end this section by discussing the more common variants of ``club" and 
``stationary" subsets of $P_\kappa X$, introduced by Jech in 
\cite{MR0325397}. Recall that, for a regular cardinal $\kappa$ and 
a set $X \supseteq \kappa$, a set $C \subseteq P_\kappa X$ is said to be 
\emph{club} in $P_\kappa X$ if it is $\subseteq$-cofinal in $P_\kappa X$ 
and, whenever $D \subseteq C$ is a $\subseteq$-linearly ordered 
set of cardinality less than $\kappa$, we have $\bigcup D 
\in C$. This latter requirement is equivalent to the following formal 
strengthening: whenever $D \subseteq C$ is $\subseteq$-directed 
and $|D| < \kappa$, we have $\bigcup D \in C$. 
We then say that a set $S \subseteq P_\kappa X$ is \emph{stationary} if, 
for every club $C$ in $P_\kappa X$, we have $S \cap C \neq \emptyset$.
The following basic observation justifies the use of the name ``weak club" 
for the notion thusly designated above.

\begin{proposition}
    If $\kappa$ is weakly inaccessible, $X \supseteq \kappa$ is a set of 
    ordinals, and $C$ is a club in $P_\kappa X$, then $C$ is a weak 
    club in $P_\kappa X$.
\end{proposition}

\begin{proof}
    Suppose that $C$ is a club in $P_\kappa X$. Since $\kappa$ is a limit 
    cardinal, the fact that $C$ is $\subseteq$-cofinal implies that it is 
    also $\prec$-cofinal. To verify closure, fix 
    $x \in P_\kappa X$ such that $C$ is $\prec$-cofinal in $P_{\kappa_x} 
    x$. Then it is straightforward to construct a set $D \subseteq C 
    \cap P_{\kappa_x} x$ such that
    \begin{itemize}
        \item $D$ is $\subseteq$-directed;
        \item $|D| \leq |x| < \kappa$;
        \item $\bigcup D = x$.
    \end{itemize}
    Since $C$ is a club, it follows that $\bigcup D = x \in C$. Thus, 
    $C$ is a weak club.
\end{proof}

\section{Two-cardinal derived topologies and $\xi$-strong stationarity}

Fix for this section an arbitrary regular uncountable cardinal $\kappa$ and 
a set of ordinals $X \supseteq \kappa$. We will investigate a sequence of 
derived topologies $\langle \tau_\xi \mid \xi < \kappa \rangle$ 
on $P_\kappa X$, simultaneously isolating a hierarchy of stationary 
reflection principles that characterize the existence of limit points 
with respect to these topologies. We emphasize that all definitions and 
arguments in this section are in the context of the ambient space of 
$P_\kappa X$. We begin by describing $\tau_0$, a generalization of 
the order topology.

\subsection{A generalization of the order topology to $P_\kappa X$}\label{section_generalizing_the_order_topology}

Given $x,y\in P_\kappa X$ with $x\strsub y$, let
\[(x,y]=\{z\in P_\kappa X\st x\strsub z\strsub y \lor z=y\}\]
and
\[(x,y)=\{z\in P_\kappa X\st x\strsub z\strsub y\}.\]
Let $\tau_0$ be the topology on $P_\kappa X$ generated by 
\[\B_0=\{0\}\cup\{(x,y]\st x,y\in P_\kappa X\land x\strsub y\}.\]
It is easy to see that $\B_0$ is a base for $\tau_0$ because if $x\in(a_0,b_0]\cap\cdots\cap (a_{n-1},b_{n-1}]$ then $x\in (a,x]\subseteq (a_0,b_0]\cap\cdots\cap (a_{n-1},b_{n-1}]$ where $a=\bigcup_{i < n}a_i$. For $A \subseteq P_\kappa X$, let
\[d_0(A)=\{x\in P_\kappa X\st \text{$x$ is a limit point of $A$ in $(P_\kappa X,\tau_0)$}\}.\]

\begin{proposition}
For every $A \subseteq P_\kappa X$, 
\[d_0(A)=\{x\in P_\kappa X\st \text{$A$ is $\strsub $-cofinal in $P_{\kappa_x}x$}\}.\]
\end{proposition}

\begin{proof}
Fix $A \subseteq P_\kappa X$ and $x\in d_0(A)$, and suppose that $y\in P_{\kappa_x}x$. Since $(y,x]$ is an open neighborhood of $x$, we can choose a $z\in (y,x]\cap A$ with $z\neq x$. This implies $z\in (y,x)\cap A$, and hence $A$ is $\strsub $-cofinal in $P_{\kappa_x}x$. Conversely, suppose $A$ is $\strsub $-cofinal in $P_{\kappa_x}x$ and let $(a,b]$ be a basic open neighborhood of $x$. Then $a\in P_{\kappa_x}x$ and we may choose some $y\in A$ with $a\strsub y\in P_{\kappa_x}x$. Hence $y\in (a,b]\cap A\setminus\{x\}$.
\end{proof}

\begin{corollary}\label{corollary_successors_are_isolated}
A point $x\in P_\kappa X$ is not isolated in $\tau_0$ if and only if $\kappa_x=|x\cap\kappa|$ is a limit cardinal.
\end{corollary}

The following proposition connects the order topology $\tau_0$ on $P_\kappa X$ to the notion of weak club discussed in Section \ref{section_strong_stationarity}, in the case where $\kappa$ is a weakly Mahlo cardinal.

\begin{proposition}\label{proposition_trace_of_a_cofinal_set}
If $\kappa$ is weakly Mahlo and $A$ is $\strsub$-cofinal in $P_\kappa X$ then $d_0(A)$ is a weak club in $P_\kappa X$.
\end{proposition}

\begin{proof}
First let us show that $d_0(A)$ is $\strsub$-closed. Suppose $d_0(A)$ is $\strsub$-cofinal in $P_{\kappa_x}x$. Fix $a\in P_{\kappa_x}x$ and let $b\in d_0(A)\cap(a,x)$. Then $A$ is $\strsub$-cofinal in $P_{\kappa_b}b$, so we may choose $c\in A\cap (a,b)\subseteq A\cap (a,x)$ and hence $x\in d_0(A)$.

Let us show that $d_0(A)$ is $\strsub$-cofinal in $P_\kappa X$. Fix $x\in P_\kappa X$. We define an increasing chain $\<x_\eta\st\eta<\kappa\>$ in $P_\kappa X$ as follows. Let $x_0=x$. Given $x_\eta$ choose $x_{\eta+1}\in A$ with $x_\eta\strsub x_{\eta+1}$. If $\eta<\kappa$ is a limit let $x_\eta=\bigcup_{\zeta<\eta}x_\zeta$. Then $\<\kappa_{x_\eta}\st\eta<\kappa\>$ is a strictly increasing sequence in $\kappa$ and the set 
\[C=\{\eta<\kappa\st (\forall\zeta<\eta)\ \kappa_{x_\zeta}<\eta\}=\{\eta<\kappa\st\kappa_{x_\eta}=\eta\}\]
is club in $\kappa$. Thus, since $\kappa$ is weakly Mahlo, we can fix some regular $\kappa_{x_\eta}=\eta\in C$. Let us show that $x_\eta \in d_0(A)$. Suppose $a\in P_{\kappa_{x_\eta}}x_\eta$. Since $\kappa_{x_\eta}=\eta$ is regular, $|a|<\kappa_{x_\eta}$ implies $a\in P_{\kappa_{x_\zeta}}x_\zeta$ for some $\zeta<\eta$, and therefore $a\in P_{\kappa_{x_{\zeta+1}}}x_{\zeta+1}$ which entails that $a\strsub x_{\zeta+1}\in A$.
\end{proof}

Recall that an ordinal $\delta$ has uncountable cofinality if and only if for every $A\subseteq\delta$ which is unbounded in $\delta$, there is an $\alpha<\delta$ such that $A$ is unbounded in $\alpha$. The following proposition is the analogous result for the notion of $\strsub$-cofinality in $P_\kappa X$ when $\kappa$ is weakly inaccessible.

\begin{proposition}\label{proposition_reflecting_cofinality}
If $\kappa$ is weakly inaccessible, then the following are equivalent.
\begin{enumerate}
\item $\kappa$ is a weakly Mahlo cardinal.
\item For all $A\subseteq P_\kappa X$ if $A$ is $\strsub $-cofinal in $P_\kappa X$ then there is an $x\in P_\kappa X$ such that $A$ is $\strsub $-cofinal in $P_{\kappa_x}x$.
\end{enumerate}
\end{proposition}

\begin{proof}
The fact that (1) implies (2) follows from Proposition \ref{proposition_trace_of_a_cofinal_set}. Let us show that (2) implies (1). We assume (2) holds, and that $\kappa$ is weakly inaccessible but not weakly Mahlo. Let $C \subseteq \kappa$ be a club consisting of 
singular cardinals, and let $D=\{x\in P_\kappa X\st x\cap\kappa\in C\}$. Then $D$ is $\strsub $-cofinal in $P_\kappa X$. By (2), there is a $y\in P_\kappa X$ such that $D$ is $\strsub $-cofinal in $P_{\kappa_y}y$. Then $\kappa_y$ is a limit cardinal and $C$ is cofinal in $\kappa_y$, so $\kappa_y\in C$ is a singular cardinal. 

Let us argue that $y\cap\kappa$ is an ordinal less than $\kappa$. Suppose $\alpha\in y\cap\kappa$. Since $D$ is cofinal in $P_{\kappa_y}y$ there is some $z\in D\cap P_{\kappa_y}y$ such that $\{\alpha\}\strsub  z$. Thus $\alpha\in z$ and $z\cap\kappa$ is an ordinal, which implies $\alpha\subseteq z\cap\kappa\subseteq y\cap\kappa$. Hence $y\cap\kappa$ is transitive.

Let $a\subseteq\kappa_y$ be cofinal in $\kappa_y$ with $|a|=\cf(\kappa_y)<\kappa_y$. Since $y\cap\kappa$ is an ordinal we have $a\subseteq\kappa_y=|y\cap\kappa|\subseteq y\cap\kappa$ and thus $a\in P_{\kappa_y}y$. However, there is no $x\in D\cap P_{\kappa_y}y$ with $a\strsub x$ because for such an $x$, $\kappa\cap x\in C$ would be an ordinal containing the set $a$ which is cofinal in $\kappa_y$, and hence $\kappa_x\geq\kappa_y$.
\end{proof}

We note that the assumption that $\kappa$ is weakly inaccessible 
is necessary in Proposition \ref{proposition_reflecting_cofinality}, 
but only for the somewhat trivial reason that, if $\kappa$ is a 
successor cardinal, then there are no $\prec$-cofinal subsets 
of $P_\kappa X$.

\subsection{Definitions of derived topologies and iterated stationarity in $P_\kappa X$}\label{section_derived_topologies}

With the topology $\tau_0$ on $P_\kappa X$, the base $\B_0$ for $\tau_0$ and the Cantor derivative $d_0$ in hand, we can now define the derived topologies on $P_\kappa X$ as follows. Given $\tau_\xi$, $\B_\xi$ and $d_\xi$, we let
\[\B_{\xi+1}=\B_\xi\cup\{d_\xi(A)\st A\subseteq P_\kappa X\},\]
we let $\tau_{\xi+1}$ be the topology generated by $\B_{\xi+1}$ and we let $d_{\xi+1}$ be defined by
\[d_{\xi+1}(A)=\{x\in P_\kappa X\st \text{$x$ is a limit point of $A$ in $(P_\kappa X,\tau_{\xi+1})$}\}\] 
for $A\subseteq P_\kappa X$.
When $\xi$ is a limit ordinal we let $\tau_\xi$ be the topology generated by $\B_\xi:=\bigcup_{\zeta<\xi}\B_\zeta$ and we let $d_\xi$ be the Cantor derivative of the space $(P_\kappa X,\tau_\xi)$.

Since $\B_0$ is a base for $\tau_0$, it easily follows that the sets of the form
\[I\cap d_{\xi_0}(A_0)\cap\cdots\cap d_{\xi_{n-1}}(A_{n-1})\]
where $I\in \B_0$, $n<\omega$, $\xi_i<\xi$ and $A_i\subseteq P_\kappa X$ for $i<n$, form a base for $\tau_\xi$. We return to the question of whether 
or not $\B_\xi$ forms a base for $\tau_\xi$ in Theorem \ref{theorem_base_characterization} below, as well as in Subsection \ref{section_variation}.

Let us note here that the next two lemmas can easily be established using arguments similar to those for \cite[Proposition 2.1 and Corollary 2.2]{MR3894041}.

\begin{lemma}\label{lemma_factor_out_d_xi}
For all $\zeta<\xi$ and all $A_0,\ldots,A_n\subseteq P_\kappa X$,
\[d_\zeta(A_0)\cap\cdots\cap d_\zeta(A_{n-1})\cap d_\xi(A_n)=d_\xi(d_\zeta(A_0)\cap\cdots\cap d_\zeta(A_{n-1})\cap A_n).\]
\end{lemma}

\begin{lemma}\label{lemma_base}
For every ordinal $\xi$, the sets of the form
\[I\cap d_\zeta(A_0)\cap\cdots\cap d_\zeta(A_{n-1})\]
where $I\in \B_0$, $n<\omega$, $\zeta<\xi$ and $A_i\subseteq P_\kappa X$ for $i<n$, form a base for $\tau_\xi$.
\end{lemma}

In the next few sections, we will characterize the non-isolated points of the spaces $(P_\kappa X,\tau_\xi)$ in terms of the following two-cardinal notions of $\xi$-s-strong stationarity.

\begin{definition}\label{definition_xi_s_stationary}
\begin{enumerate}
\item For $A\subseteq P_\kappa X$ and $x\in P_\kappa X$, we say that $A$ is \emph{$0$-strongly stationary in $P_{\kappa_x}x$} if and only if $A$ is $\strsub$-cofinal in $P_{\kappa_x}x$. For an ordinal $\xi>0$, we say that $A$ is $\xi$-strongly stationary in $P_{\kappa_x}x$ if and only if $\kappa_x$ is a limit cardinal\footnote{The requirement that $\kappa_x$ be a limit cardinal in order for $A$ to be $\xi$-strongly stationary in $P_{\kappa_x}x$ is necessary because otherwise, when $\kappa_x$ is a successor ordinal there are \emph{no} $0$-strongly stationary subsets of $P_{\kappa_x}x$ and hence \emph{every} subset of $P_{\kappa_x}x$ would be $1$-strongly stationary.} and, whenever $\zeta < \xi$ and $S\subseteq P_{\kappa_x}x$ is $\zeta$-strongly stationary in $P_{\kappa_x}x$, there is some $y\in A\cap P_{\kappa_x}x$ such that $S$ is $\zeta$-strongly stationary in $P_{\kappa_y}y$.
\item A set $C\subseteq P_\kappa X$ is called \emph{$0$-weak club in $P_{\kappa_x}x$} if and only if it is $\strsub$-cofinal and $\strsub$-closed in $P_{\kappa_x}x$. For an ordinal $\xi>0$, we say that $C$ is \emph{$\xi$-weak club in $P_{\kappa_x}x$} if and only if it is $\xi$-strongly stationary in $P_{\kappa_x}x$ and it is \emph{$\xi$-strongly stationary closed in $P_{\kappa_x}x$}, meaning that whenever $y \strsub x$ and $C$ is $\xi$-strongly stationary in $P_{\kappa_y}y$ we have $y\in C$.
\item We say that $A$ is \emph{$0$-s-strongly stationary in $P_{\kappa_x}x$} if and only if $A$ is $\strsub$-cofinal in $P_{\kappa_x}x$. For an ordinal $\xi>0$, we say that $A$ is \emph{$\xi$-s-strongly stationary in $P_{\kappa_x}x$} if and only if $\kappa_x$ is a limit ordinal and, whenever $\zeta<\xi$, $\kappa_x$ and $S,T\subseteq P_\kappa X$ are $\zeta$-s-strongly stationary in $P_{\kappa_x}x$ there is some $y\in A\cap P_{\kappa_x}x$ such that $S$ and $T$ are both $\zeta$-s-strongly stationary in $P_{\kappa_y}y$.
\item A set $C\subseteq P_\kappa X$ is called \emph{$0$-s-weak club in $P_{\kappa_x}x$} if and only if it is $0$-s-strongly stationary in $P_{\kappa_x}x$ and whenever $y \prec x$ and $C$ is $0$-s-strongly stationary in $P_{\kappa_y}y$ we have $y\in C$. For an ordinal $\xi>0$, we say that $C$ is \emph{$\xi$-s-weak club in $P_{\kappa_x}x$} if and only if it is $\xi$-s-strongly stationary in $P_{\kappa_x}x$ and it is \emph{$\xi$-s-closed in $P_{\kappa_x}x$}, meaning that whenever $y<x$ and $C$ is $\xi$-s-strongly stationary in $P_{\kappa_y}y$ we have $y\in C$.
\end{enumerate}
\end{definition}

In what follows, given $x \in P_\kappa X$ and $\xi < \kappa$, we will 
simply say that, e.g., $P_{\kappa_x}x$ is $\xi$-s-strongly stationary to mean 
that it is $\xi$-s-strongly stationary \emph{in $P_{\kappa_x} x$}.
Let us first note the following simple proposition, which justifies 
the restriction of our attention to values of $\xi$ less than $\kappa$.
By the results of Subsection \ref{section_indescribability}, 
the proposition is sharp, at least assuming the consistency of 
certain large cardinals.

\begin{proposition}\label{proposition_kappa_x_plus_1}
    For all $x \in P_\kappa X$, $P_{\kappa_x} x$ is not 
    $(\kappa_x+1)$-strongly stationary.
\end{proposition}

\begin{proof}
    Suppose otherwise, and let $x \in P_\kappa X$ be a counterexample 
    such that $\kappa_x$ is minimal among all counterexamples. Since 
    $P_{\kappa_x} x$ is $(\kappa_x+1)$-strongly stationary, it is 
    \emph{a fortiori} $\kappa_x$-strongly stationary. Therefore, by 
    the definition of $(\kappa_x+1)$-strongly stationary, we can find 
    $y \in P_{\kappa_x} x$ such that $P_{\kappa_x} x$ is 
    $\kappa_x$-strongly stationary in $P_{\kappa_y} y$. Since 
    $\kappa_x > \kappa_y$, this implies that $P_{\kappa_y} y$ is 
    $(\kappa_y + 1)$-strongly stationary, contradicting the minimality 
    of $\kappa_x$.
\end{proof}
Considering the previous proposition, it is natural to wonder whether the definitions of $\xi$-strong stationarity and $\xi$-s-strong stationarity can be modified using canonical functions to allow for settings in which some $x\in P_\kappa X$ can be $\xi$-strongly stationary for $\kappa_x<\xi<|x|^+$; this was done in the cardinal setting by the first author in \cite{CodyHigherIndescribability}. See the discussion before Question \ref{question_higher1} and Question \ref{question_higher2} for more information.



Definition \ref{definition_xi_s_stationary} leads naturally to the definition of the following ideals, which can be strongly normal under a certain large cardinal hypothesis by Proposition \ref{proposition_strongly_normal}.

\begin{definition}\label{definition_xi_s_ideal}
Suppose that $x \in P_\kappa X$. We define 
\[\overline{\NS}_{\kappa_x,x}^\xi=\{A\subseteq P_\kappa X\st \textrm{$A$ is not $\xi$-strongly stationary in $P_{\kappa_x} x$}\}\]
and
\[\NS_{\kappa_x,x}^\xi=\{A\subseteq P_\kappa X\st \textrm{$A$ is not $\xi$-s-strongly stationary in $P_{\kappa_x} x$}\}.\]
\end{definition}

Let us show that for the $x$'s in $P_\kappa X$ that we will care most about, namely those for which $\kappa_x$ is regular, $1$-strong stationarity and $1$-s-strong stationarity are equivalent in $P_{\kappa_x}x$; moreover, if $\kappa_x$ is inaccessible, then these notions are equivalent to strong stationarity in $P_{\kappa_x}x$ plus the Mahloness of $\kappa_x$.

\begin{proposition}\label{proposition_strong_vs_1_s_stationarity}
Suppose $A \subseteq P_\kappa X$ and $x\in P_\kappa X$ with $\kappa_x$ regular. Then the following are equivalent, and both imply that 
$\kappa_x$ is weakly Mahlo.
\begin{enumerate}
    \item $A$ is $1$-strongly stationary in $P_{\kappa_x}x$.
    \item $A$ is $1$-s-strongly stationary in $P_{\kappa_x}x$.
\end{enumerate}
If, moreover, $\kappa_x$ is strongly inaccessible, then these two statements 
are also equivalent to the following:
\begin{enumerate}
    \setcounter{enumi}{2}
    \item $\kappa_x$ is Mahlo and $A$ is strongly stationary in $P_{\kappa_x}x$.
\end{enumerate}
\end{proposition}

\begin{proof}
    Note that, if $A$ is $1$-strongly stationary in $P_{\kappa_x}x$, 
    then $\kappa_x$ is a limit cardinal and hence weakly inaccessible. 
    We can thus assume that this is the case.
    (2) $\implies$ (1) is trivial. Let us now assume that $A$ is 
    $1$-strongly stationary in $P_{\kappa_x} x$. By Proposition \ref{proposition_reflecting_cofinality}, it follows that $\kappa_x$ is weakly Mahlo. To see that $A$ is $1$-s-strongly stationary in 
    $P_{\kappa_x}x$, fix sets $S_0,S_1 \subseteq P_\kappa X$ that are both 
    $\prec$-cofinal in $P_{\kappa_x} x$. Let $T$ be the set of 
    $y \in P_{\kappa_x} x$ such that $S_0$ and $S_1$ are both $\prec$-cofinal 
    in $P_{\kappa_y} y$. We claim that $T$ is $\prec$-cofinal in 
    $P_{\kappa_x} x$. To see this, fix an arbitrary $y_0 \in P_{\kappa_x} x$.
    Define a continuous, $\prec$-increasing sequence $\langle y_\eta \mid 
    \eta < \kappa_x \rangle$ in $P_{\kappa_x} x$ as follows. The set 
    $y_0$ is already fixed. Given $y_\eta$, find $z^0_\eta \in S_0$ 
    and $z^1_\eta \in S_1$ such that, for all $i < 2$, we have $y_\eta 
    \prec z^i_\eta \prec x$. Then let $y_{\eta + 1} = z^0_\eta \cup z^1_\eta$. 
    If $\xi < \kappa_x$ is a limit ordinal, let $y_\xi = \bigcup \{y_\eta 
    \mid \eta < \xi\}$. The set of $\eta < \kappa_x$ for which  
    $\kappa_{y_\eta} = \eta$ is club in $\kappa_x$, so, since $\kappa_x$ 
    is weakly Mahlo, we can fix some regular cardinal $\eta < \kappa_x$ 
    such that $\kappa_{y_\eta} = \eta$. A now-familiar argument then shows 
    that $S_0$ and $S_1$ are both $\prec$-cofinal in $y_\eta$, and hence 
    $y_\eta \in T$.

    Since $A$ is $1$-strongly stationary in $P_{\kappa_x}x$, we can find 
    $w \in A$ such that $T$ is $\prec$-cofinal in $P_{\kappa_w}w$. It follows 
    immediately that $S_0$ and $S_1$ are both $\prec$-cofinal in $P_{\kappa_w} 
    w$; therefore, $A$ is $1$-s-strongly stationary in $P_{\kappa_x} x$.

    For the ``moreover" clause, assume that $\kappa_x$ is strongly 
    inaccessible and $A$ is $1$-strongly stationary in $P_{\kappa_x} x$. 
    The fact that $\kappa_x$ is Mahlo follows from the previous paragraphs. 
    To show that $A$ is strongly stationary in $P_{\kappa_x} x$, suppose $C$ is a weak club subset of $P_{\kappa_x}x$. Since $A$ is $1$-strongly stationary there is some $y\in A$ such that $C$ is $\strsub$-cofinal in $P_{\kappa_y}y$. Since $C$ is weakly closed we have $y\in A\cap C$.
\end{proof}

\subsection{The $\tau_1$ topology on $P_\kappa X$}

We now discuss the \emph{first derived topology $\tau_1$ on $P_\kappa X$}. Recall that this is the topology generated by 
\[\B_1=\B_0\cup\{d_0(A)\st A\subseteq P_\kappa X\}.\]

\begin{remark}\label{remark_fix}
By definition $\B_1$ is a subbase for the first derived topology $\tau_1$ on $P_\kappa X$, but it is not clear whether it is a base for $\tau_1$ (essentially because of Proposition \ref{proposition_reflecting_cofinality}). Recall that the subbase for the first derived topology on an ordinal $\delta$ is always a base for that topology (see \cite{MR3894041}). This difference seems not to create too much difficulty so we proceed with our definition as is, but in Section \ref{section_variation} we show that, if we pass to 
a certain club subset $C$ of $P_\kappa X$, then the natural restriction 
of $\B_1$ to $C$ is a base for the subspace topology on $C$ induced by 
$\tau_1$.
\end{remark}

We will need the following lemma.

\begin{lemma}\label{lemma_2_to_n_for_1_stationarity}
Fix $x \in P_\kappa X$, and 
suppose that $A$ is $1$-s-strongly stationary in $P_{\kappa_x}x$ and $A_0,\ldots,A_{n-1}$ are all $0$-s-strongly stationary (i.e.\ $\strsub$-cofinal) in $P_{\kappa_x}x$, where $n\geq 2$. Then $d_0(A_0)\cap\cdots\cap d_0(A_{n-1})\cap A$ is $1$-s-strongly stationary in $P_{\kappa_x}x$.
\end{lemma}

\begin{proof}
First let us use a straightforward inductive argument on $n\geq 2$ to show that whenever $A_0,\ldots,A_{n-1}$ are $0$-strongly stationary in $P_{\kappa_x}x$, the set $d_0(A_0)\cap\cdots\cap d_0(A_{n-1})\cap A$ is $0$-strongly stationary in $P_{\kappa_x}x$. Suppose $A_0$ and $A_1$ are $\strsub$-cofinal in $P_{\kappa_x}x$ and note that $\kappa_x$ must be a limit cardinal. To show that $d_0(A_0)\cap d_0(A_1)\cap A$ is $\strsub$-cofinal in $P_{\kappa_x}x$, fix $y\in P_{\kappa_x}x$. Then $A_0\cap (y,x)$ and $A_1\cap (y,x)$ are $\strsub$-cofinal in $P_{\kappa_x}x$. Since $A$ is $1$-s-strongly stationary in $P_{\kappa_x}x$ there is an $a\in A\cap P_{\kappa_x}x$ such that $A_0\cap (y,x)$ and $A_1\cap (y,x)$ are both $\strsub$-cofinal in $P_{\kappa_a}a$, and hence $y<a$. Therefore $a\in d_0(A_0)\cap d_0(A_1)\cap A\cap (y,x)$. Now suppose the result holds for $n$, and suppose $A_0,\ldots,A_{n-1},A_n$ are all $\strsub$-cofinal in $P_{\kappa_x}x$. By our inductive hypothesis, $d_0(A_0)\cap\cdots\cap d_0(A_{n-1})$ is $\strsub$-cofinal in $P_{\kappa_x}x$, and by the base case the set $d_0(d_0(A_0)\cap\cdots\cap d_0(A_{n-1}))\cap d_0(A_n)\cap A$ is $\strsub$-cofinal in $P_{\kappa_x}x$. But 
\begin{align*}
    d_0(d_0(A_0)\cap\cdots\cap d_0(A_{n-1}))\cap d_0(A_n) &\subseteq d_0(d_0(A_0))\cap \cdots\cap d_0(d_0(A_{n-1}))\cap d_0(A_n) \\
    &\subseteq d_0(A_0)\cap \cdots\cap d_0(A_{n-1})\cap d_0(A_n).
\end{align*}

Now we prove the statement of the lemma. Fix sets $A_0,\ldots,A_{n-1}\subseteq P_\kappa X$ that are $\strsub$-cofinal in $P_{\kappa_x}x$. To show that $d_0(A_0)\cap\cdots\cap d_0(A_{n-1})\cap A$ is $1$-s-strongly stationary in $P_{\kappa_x}x$, fix sets $S$ and $T$ that are $\strsub$-cofinal in $P_{\kappa_x}x$. By the previous paragraph, it follows, that the set
\[d_0(S)\cap d_0(T)\cap d_0(A_0)\cap\cdots\cap d_0(A_{n-1})\cap A\]
is $\strsub$-cofinal in $P_{\kappa_x}x$ and hence there is some
\[y\in d_0(S)\cap d_0(T)\cap d_0(A_0)\cap\cdots\cap d_0(A_{n-1})\cap A,\]
which establishes that $d_0(A_0)\cap\cdots\cap d_0(A_{n-1})\cap A$ is $1$-s-strongly stationary.
\end{proof}

\begin{corollary}
Suppose $P_{\kappa_x}x$ is $1$-s-strongly stationary. Then a set $A$ is $1$-s-strongly stationary in $P_{\kappa_x}x$ if and only if for all sets $C$ which are $0$-s-weak club in $P_{\kappa_x}x$ we have $A\cap C\cap P_{\kappa_x}x\neq\emptyset$.
\end{corollary}

\begin{proof}
Suppose $A$ is $1$-s-strongly stationary in $P_{\kappa_x}x$ and $C$ is $0$-s-weak club in $P_{\kappa_x}x$. Then $d_0(C)\cap P_{\kappa_x}x\subseteq C\cap P_{\kappa_x}x$ and by Lemma \ref{lemma_2_to_n_for_1_stationarity}, $d_0(C)\cap A$ is $1$-s-strongly stationary in $P_{\kappa_x}x$. Thus $A\cap C\cap P_{\kappa_x}x\neq\emptyset$. Conversely, assume that $A \cap C\cap P_{\kappa_x}x\neq\emptyset$ whenever $C$ is $0$-s-weak club in $P_{\kappa_x}x$. Fix sets $S$ and $T$ that are $0$-s-strongly stationary in $P_{\kappa_x}x$. Then $d_0(S)\cap d_0(T)$ is $0$-s-weak club in $P_{\kappa_x}x$ because $d_0(S)\cap d_0(T)\cap P_{\kappa_x}x$ is $1$-s-strongly stationary and hence $0$-s-strongly stationary in $P_{\kappa_x}x$ by Lemma \ref{lemma_2_to_n_for_1_stationarity}, and $d_0(S)\cap d_0(T)$ is $0$-s-closed in $P_{\kappa_x}x$ since 
\[d_0(d_0(S)\cap d_0(T))\subseteq d_0(S)\cap d_0(T)\]
as a consequence of the fact that $d_0$ is the limit point operator of the space $(P_\kappa X,\tau_0)$.
\end{proof}

\begin{proposition}\label{proposition_limit_points_tau_1}
If $A\subseteq P_\kappa X$ then
\[d_1(A)=\{x\in P_\kappa X\st \text{$A$ is $1$-s-strongly stationary in $P_{\kappa_x}x$}\}.\]
\end{proposition}

\begin{proof}
Suppose $A$ is not $1$-s-strongly stationary in $P_{\kappa_x}x$. If $\kappa_x$ is a successor cardinal then $x$ is isolated in $(P_\kappa X,\tau_1)$ by Corollary \ref{corollary_successors_are_isolated} and hence $x\notin d_1(A)$. Suppose $\kappa_x$ is a limit cardinal. Then there are sets $S$ and $T$ which are $0$-strongly stationary in $P_{\kappa_x}x$ such that $d_0(S)\cap d_0(T)\cap A\cap P_{\kappa_x}x=\emptyset$. Then it follows that $d_0(S)\cap d_0(T)\cap (0,x]$ is an open neighborhood of $x$ in the $\tau_1$ topology that does not intersect $A$ in some point other than $x$. Hence $x\notin d_1(A)$.

Conversely, suppose $A$ is $1$-s-strongly stationary in $P_{\kappa_x}x$. Suppose $x\in I\cap d_0(A_0)\cap\cdots\cap d_0(A_{n-1})$. Then the sets $A_0,\ldots,A_{n-1}$ are all $\strsub$-cofinal in $P_{\kappa_x}x$, and by Lemma \ref{lemma_2_to_n_for_1_stationarity}, the set $I\cap d_0(A_0)\cap \cdots d_0(A_{n-1})\cap A$ is $\strsub$-cofinal in $P_{\kappa_x}x$, which implies that the open neighborhood $I\cap d_0(A_0)\cap\cdots\cap d_0(A_{n-1})$ of $x$ intersects $A$ in some point other than $x$.
\end{proof}

\begin{corollary}
    A point $x\in P_\kappa X$ is not isolated in $(P_\kappa X,\tau_1)$ if and only if $P_{\kappa_x}x$ is $1$-s-strongly stationary in $P_{\kappa_x}x$.
\end{corollary}

\subsection{The $\tau_\xi$ topology on $P_\kappa X$ for $\xi \geq 2$}

We now move to the general setting.
Let us first characterize limit points of sets in the spaces $(P_\kappa X,\tau_\xi)$ in terms of $\xi$-s-strong stationarity.

\begin{theorem}\label{theorem_induction}
For all $\xi<\kappa$ the following hold.
\begin{enumerate}
\item[$(1)_\xi$] We have
\[d_\xi(A)=\{x\in P_\kappa X\st \textrm{$A$ is $\xi$-s-strongly stationary in $P_{\kappa_x}x$}\}.\]
\item[$(2)_\xi$] For all $x\in P_\kappa X$, a set $A$ is $\xi+1$-s-strongly stationary in $P_{\kappa_x}x$ if and only if for all $\zeta\leq\xi$ and every pair $S,T$ of subsets of $P_{\kappa_x}x$ that are $\zeta$-s-strongly stationary in $P_{\kappa_x}x$, we have $A\cap d_\zeta(S)\cap d_\zeta(T)\neq\emptyset$ (equivalently $A\cap d_\zeta(S)\cap d_\zeta(T)$ is $\zeta$-s-strongly stationary in $P_{\kappa_x}x$).
\item[$(3)_\xi$] For all $x\in P_\kappa X$, if $A$ is $\xi$-s-strongly stationary in $P_{\kappa_x}x$ and $A_i$ is $\zeta_i$-s-strongly stationary in $P_{\kappa_x}x$ for some $\zeta_i<\xi$ and all $i<n$, then $A\cap d_{\zeta_0}(A_0)\cap\cdots\cap d_{\zeta_{n-1}}(A_{n-1})$ is $\xi$-s-strongly stationary in $P_{\kappa_x}x$.
\end{enumerate}
\end{theorem}

\begin{proof}
We have already established that $(1)_\xi$, $(2)_\xi$ and $(3)_\xi$ hold for $\xi\leq 1$. Given these base cases, the fact that (1), (2) and (3) hold for all $\xi<\kappa$ can be established by simultaneous induction using an argument which is essentially identical to that of \cite[Proposition 2.10]{MR3894041}. For the reader's convenience, we include the argument here.

First, suppose $(1)_\zeta$, $(2)_\zeta$ and $(3)_\zeta$ hold for all $\zeta$ less than some limit ordinal $\xi<\kappa$. It is clear that $(1)_\xi$ and $(3)_\xi$ also must hold. Let us prove that $(2)_\xi$ holds. Notice that the backward direction of $(2)_\xi$ easily follows from the definition of $\xi+1$-s-strong stationarity and the fact that $(1)_\zeta$ holds for $\zeta\leq \xi$. For the forward direction of $(2)_\xi$, suppose $A$ is $\xi+1$-s-strongly stationary in $P_{\kappa_x}x$. Fix $\zeta\leq\xi$ and a pair $S,T$ of $\zeta$-s-strongly stationary subsets of $P_{\kappa_x}x$. To show that $A\cap d_\zeta(S)\cap d_\zeta(T)$ is $\zeta$-s-strongly stationary in $P_{\kappa_x}x$, fix sets $A,B$ that are $\eta$-s-strongly stationary in $P_{\kappa_x}x$ where $\eta<\zeta$. Using the fact that (3) holds for $\zeta$, we see that $S\cap d_\eta(A)\cap d_\eta(B)$ is $\zeta$-s-strongly stationary in $P_{\kappa_x}x$. Since $A$ is $\xi+1$-s-strongly stationary, and applying the fact that $(1)_\zeta$ holds, we have 
\[\emptyset\neq d_\zeta(d_\eta(A)\cap d_\eta(B)\cap S)\cap d_\zeta(T)\cap A.\]
But, by Lemme \ref{lemma_factor_out_d_xi},
\[d_\zeta(d_\eta(A)\cap d_\eta(B)\cap S)\cap d_\zeta(T)\cap A = d_\eta(A)\cap d_\eta(B)\cap d_\zeta(S)\cap d_\zeta(T)\cap A.\]
Thus, $d_\zeta(S)\cap d_\zeta(T)\cap A$ is $\zeta$-s-strongly stationary in $P_{\kappa_x}x$.

One can show that if $(1)_{\leq\xi}$, $(2)_{\leq\xi}$ and $(3)_{\leq\xi}$ hold then, by induction on $n$, $(3)_{\xi+1}$ must also hold. For the reader's convenience we provide a proof that $(3)_{\xi+1}$ holds for $n=1$, the remaining case is the same as \cite[Proposition 2.10]{MR3894041}. Suppose $n=1$. To prove that $A\cap d_{\zeta_0}(A_0)$ is $\xi+1$-s-strongly stationary in $P_{\kappa_x}x$, fix sets $S$ and $T$ that are $\eta$-s-strongly stationary in $P_{\kappa_x}x$ for some $\eta\leq\xi$. By $(1)_{\leq\xi}$, it will suffice to show that $A\cap d_{\zeta_0}(A_0)\cap d_\eta(S)\cap d_\eta(T)\neq\emptyset$. If ${\zeta_0}=\eta$, then by $(2)_{\leq\xi}$, it follows that set $A\cap d_{\zeta_0}(A_0)\cap d_\eta(d_\eta(S)\cap d_\eta(T))$, which is contained in $A\cap d_{\zeta_0}(A_0)\cap d_\eta(S)\cap d_\eta(T)$, is ${\zeta_0}$-s-strongly stationary in $P_{\kappa_x}x$, and thus $A\cap d_{\zeta_0}(A_0)\cap d_\eta(S)\cap d_\eta(T)\neq\emptyset$. If ${\zeta_0}<\eta$, then by $(3)_\eta$, if follows that $d_{\zeta_0}(A_0)$ is $\eta$-s-strongly stationary in $P_{\kappa_x}x$, and by $(2)_{\xi}$, the set $A\cap d_\eta(d_{\zeta_0}(A_0))\cap d_\eta(d_\eta(S)\cap d_\eta(T))$, which is contained in $A\cap d_{\zeta_0}(A_0)\cap d_\eta(S)\cap d_\eta(T)$, is $\eta$-s-strongly stationary in $P_{\kappa_x}x$. If $\zeta_0>\eta$ then by $(2)_\xi$ the set $A\cap d_{\zeta_0}(A_0)$ is $\zeta_0$-s-strongly stationary in $P_{\kappa_x}x$ and thus $A\cap d_{\zeta_0}(A_0)\cap d_\eta(S)\cap d_\eta(T)\neq\emptyset$.

Let us prove that if $(1)_{\leq\xi}$, $(2)_{\leq\xi}$ and $(3)_{\leq\xi+1}$ hold then $(1)_{\xi+1}$ holds (this argument is similar to that of Proposition \ref{proposition_limit_points_tau_1}). Suppose $A$ is not $\xi+1$-s-strongly stationary in $P_{\kappa_x}x$. Then by $(1)_{\leq\xi}$, there there are sets $S$ and $T$ which are $\zeta$-s-strongly stationary in $P_{\kappa_x}x$ such that $A\cap d_\zeta(S)\cap d_\zeta(T)=\emptyset$. Thus $d_\zeta(S)\cap d_\zeta(T)\cap (0,x]$ is an open neighborhood of $x$ in the $\tau_{\xi+1}$ topology that does not intersect $A$ in some point other than $x$. Conversly, suppose $A$ is $\xi+1$-s-strongly stationary in $P_{\kappa_x}x$. To show that $x\in d_{\xi+1}(A)$, let $U$ be an arbitrary basic open neighborhood of $x$ in the $\tau_{\xi+1}$ topology. By Lemma \ref{lemma_base}, we can assume that $U$ is of the form 
\[U=I\cap d_\zeta(A_0)\cap\cdots\cap d_\zeta(A_{n-1})\]
where $I\in \B_0$, $n<\omega$, $\zeta<\xi+1$ and $A_i\subseteq P_\kappa X$ for $i<n$. Since $x\in U$ it follows from $(1)_\zeta$ that each $A_i$ is $\zeta$-s-strongly stationary in $P_{\kappa_x}x$, and thus by $(3)_{\xi+1}$ we see that $A\cap d_\zeta(A_0)\cap\cdots\cap d_\zeta(A_{n-1})$ is $\xi+1$-s-strongly stationary in $P_{\kappa_x}x$, and thus $U$ intersects $A$ in some point other than $x$.


Finally, we prove that if $(1)_{\leq\xi+1}$, $(2)_{\leq\xi}$ and $(3)_{\leq\xi+1}$ hold, then $(2)_{\xi+1}$ must also hold.
Suppose $A$ is $\xi+2$-s-strongly stationary in $P_{\kappa_x}x$. 
By $(2)_{\leq\xi}$, it suffices to show that whenever $S$ and $T$ are $\xi+1$-s-strongly stationary in $P_{\kappa_x}x$, 
the set $A\cap d_{\xi+1}(S)\cap d_{\xi+1}(T)$ is $\xi+1$-s-strongly stationary in $P_{\kappa_x}x$. So, fix $Y$ and $Z$ 
which are $\zeta$-s-strongly stationary 
in $P_{\kappa_x}x$ for some $\zeta\leq \xi$. 
By $(3)_{\xi+1}$, it follows that $S\cap d_\zeta(R)$ and $T\cap d_\zeta(Z)$ are $\xi+1$-s-strongly stationary in $P_{\kappa_x}x$, and hence by the $\xi+1$-s-strong stationarity of $A$ and $(1)_{\xi+1}$ we have $A\cap d_{\xi+1}(A\cap d_\zeta(R))\cap d_{\xi+1}(B\cap d_\zeta(Z))\neq\emptyset$. But 
\[d_{\xi+1}(A\cap d_\zeta(R))\cap d_{\xi+1}(B\cap d_\zeta(Z))=d_{\xi+1}(S)\cap d_{\xi+1}(T)\cap d_\zeta(R)\cap d_\zeta(Z)\]
by Lemma \ref{lemma_factor_out_d_xi}, and thus $A\cap d_{\xi+1}(S)\cap d_{\xi+1}(T)$ is $\xi+1$-s-strongly stationary in $P_{\kappa_x}x$.
The backward direction of $(2)_{\xi+1}$ follows easily from $(1)_{\leq\xi}$.
\end{proof}

\begin{corollary} \label{corollary_weak_club}
Suppose $P_{\kappa_x}x$ is $\xi$-s-strongly stationary where $\xi \leq \kappa_x$ and $A$ is $\zeta$-s-strongly stationary in $P_{\kappa_x}x$ for some $\zeta < \xi$. Then, for all $\zeta \leq \zeta' \leq \xi$, $d_\zeta(A)$ is $\zeta'$-s-weak club in $P_{\kappa_x}x$.
\end{corollary}

\begin{proof}
Fix $\zeta'$ with $\zeta \leq \zeta' \leq \xi$.
It follows from Theorem \ref{theorem_induction}(3) that $d_\zeta(A)$ is $\xi$-s-strongly stationary and hence $\zeta'$-s-strongly stationary 
in $P_{\kappa_x} x$. Furthermore, $d_\zeta(A)$ is $\zeta'$-s-closed below $P_{\kappa_x}x$ since $d_{\zeta'}(d_\zeta(A))\subseteq d_\zeta(d_\zeta(A))\subseteq d_\zeta(A)$.
\end{proof}

\begin{corollary}\label{corollary_xi_plus_1_s_filter_base}
Suppose $P_{\kappa_x}x$ is $\xi+1$-s-strongly stationary. Then a set $A$ is $\xi+1$-s-strongly stationary in $P_{\kappa_x}x$ if and only if $A\cap C\neq\emptyset$ for all sets $C\subseteq P_{\kappa_x}x$ which are $\xi$-s-weak club in $P_{\kappa_x}x$. Thus the filter generated by the $\xi$-s-weak club subsets of $P_{\kappa_x}x$ is the filter dual to $\NS_{\kappa_x,x}^{\xi+1}$.
\end{corollary}

\begin{proof}
Suppose $A$ is $\xi+1$-s-strongly stationary in $P_{\kappa_x}x$ and $C$ is $\xi$-s-weak club in $P_{\kappa_x}x$. By Theorem \ref{theorem_induction}(1), it follows that $d_\xi(C)\subseteq C$ and by Theorem \ref{theorem_induction}(3) we see that $d_\xi(C)\cap A$ is $\xi+1$-s-strongly stationary in $P_{\kappa_x}x$ and thus $C\cap A\cap P_{\kappa_x}x\neq\emptyset$.

Conversely, suppose $A\cap C\neq\emptyset$ whenever $C$ is a $\xi$-s-weak club subset of $P_{\kappa_x}x$. To show that $A$ is $\xi+1$-s-strongly stationary in $P_{\kappa_x}x$, suppose $S$ and $T$ are $\zeta$-s-strongly stationary in $P_{\kappa_x}x$ for some $\zeta\leq\xi$. Then the set $d_\zeta(S)\cap d_\zeta(T)$ is $\xi$-s-weak club in $P_{\kappa_x}x$ because it is $\xi$-s-strongly stationary in $P_{\kappa_x}x$ by Theorem \ref{theorem_induction}(3) and it is $\xi$-s-closed in $P_{\kappa_x}x$ since \[d_\xi(d_\zeta(S)\cap d_\zeta(T))\subseteq d_\zeta(d_\zeta(S)\cap d_\zeta(T))\subseteq d_\xi(S)\cap d_\xi(T).\]
Thus $A\cap d_\zeta(S)\cap d_\zeta(T)\cap P_{\kappa_x}x\neq\emptyset$, and hence $A$ is $\xi+1$-s-strongly stationary in $P_{\kappa_x}x$ as desired.
\end{proof}

\begin{corollary}
    Suppose that $x \in P_\kappa X$ and $\xi \leq \kappa_x$. Then 
    $x$ is not isolated in $(P_\kappa X, \tau_\xi)$ if and only if 
    $P_{\kappa_x} x$ is $\xi$-s-strongly stationary.
\end{corollary}

\begin{proof}
    For the forward direction, suppose that $P_{\kappa_x} x$ is not 
    $\xi$-s-strongly stationary. Then there is $\zeta < \xi$ and 
    sets $S,T \subseteq P_{\kappa_x}x$ such that $S$ and $T$ are both 
    $\zeta$-s-strongly stationary in $P_{\kappa_x} x$ but there is 
    not $y \prec x$ such that $S$ and $T$ are both $\zeta$-s-strongly 
    stationary in $P_{\kappa_y} y$. Then, by Theorem 
    \ref{theorem_induction}(1), we have $d_\zeta(S) \cap d_\zeta(T) = 
    \{x\}$, so $x$ is isolated in $(P_\kappa X, \tau_\xi)$. 

    For the converse, suppose that $P_{\kappa_x} x$ is 
    $\xi$-s-strongly stationary, and fix an interval $I \in \B_0$, 
    an $n < \omega$, ordinals $\xi_0, \ldots, \xi_{n-1} < \xi$, 
    and sets $A_0, \ldots, A_{n-1} \subseteq P_{\kappa_x} x$ such that
    \[
      x \in U := I \cap d_{\xi_0}(A_0) \cap \ldots \cap 
      d_{\xi_{n-1}}(A_{n-1}).
    \]
    Let $\zeta := \max\{\zeta_i \mid i < n\} < \xi$.
    By Corollary \ref{corollary_weak_club}, each of $I$, 
    $d_{\xi_0}(A_0)$, \ldots, $d_{\xi_{n-1}}(A_{n-1})$ is 
    $\zeta$-s-weak club in $P_{\kappa_x} x$. By Corollary 
    \ref{corollary_xi_plus_1_s_filter_base}, $U$ is also $\zeta$-s-weak 
    club in $P_{\kappa_x} x$. In particular, $U \neq \{x\}$; 
    hence, $x$ is not isolated in $P_{\kappa_x} x$.
\end{proof}

\begin{corollary}\label{corollary_filter_base}
Suppose $P_{\kappa_x}x$ is $\xi$-s-strongly stationary where $0<\xi \leq \kappa_x$. Then a set $A$ is $\xi$-s-strongly stationary in $P_{\kappa_x}x$ if and only if for all $\zeta<\xi$ we have $A\cap C\neq\emptyset$ for all sets $C\subseteq P_{\kappa_x}x$ which are $\zeta$-s-weak club in $P_{\kappa_x}x$.
\end{corollary}

\begin{proof}
Suppose $A$ is $\xi$-s-strongly stationary in $P_{\kappa_x}x$. Fix $\zeta<\xi$ and assume that $C\subseteq P_{\kappa_x}x$ is $\zeta$-s-weak club in $P_{\kappa_x}x$. Since $C$ is $\zeta$-s-strongly stationary in $P_{\kappa_x}x$ there is some $y\in d_\zeta(C)\cap A$, but since $d_\zeta(C)\subseteq C$ we have $y\in C\cap A$. Conversely, suppose that for all $\zeta<\xi$ and all $C\subseteq P_{\kappa_x}x$ that are $\zeta$-s-weak club in $P_{\kappa_x}x$ we have $A\cap C\neq\emptyset$. To show that $A$ is $\xi$-s-strongly stationary in $P_{\kappa_x}x$, suppose $S$ and $T$ are $\zeta$-s-strongly stationary in $P_{\kappa_x}x$ for some $\zeta<\xi$. Then, since we are assuming that $P_{\kappa_x}x$ is $\xi$-s-strongly stationary, it follows by Theorem \ref{theorem_induction}(3) that $d_\zeta(S)\cap d_\zeta(T)$ is $\xi$-s-strongly stationary in $P_{\kappa_x}x$. Furthermore, 
\[d_\zeta(d_\zeta(S)\cap d_\zeta(T))\subseteq d_\zeta(d_\zeta(S))\cap d_\zeta(d_\zeta(T))\subseteq d_\zeta(S)\cap d_\zeta(T),\]
which implies that $d_\zeta(S)\cap d_\zeta(T)$ is $\zeta$-s-weak club in $P_{\kappa_x}x$.
Conversely, suppose $A\cap d_\zeta(S)\cap d_\zeta(T)\neq\emptyset$ whenever $S$ and $T$ are $\zeta$-s-strongly stationary in $P_{\kappa_x}x$ for some $\zeta\leq\xi+1$. Then it easily follows by $(1)\leq\xi$ that $A$ is $\xi+2$-s-strongly stationary in $P_{\kappa_x}x$.
\end{proof}

\begin{proposition}
For $x\in P_\kappa X$ and $\xi \leq \kappa_x$, the set $P_{\kappa_x}x$ is $\xi$-s-strongly stationary if and only if $\NS_{\kappa_x,x}^\xi$ is a nontrivial ideal.
\end{proposition}

\begin{proof}
Suppose $P_{\kappa_x}x$ is $0$-s-strongly stationary. Then $\NS_{\kappa_x,x}^0$ is the ideal $I_{\kappa_x,x}$ consisting of all subsets $A$ of $P_{\kappa_x}x$ such that there is some $y\in P_{\kappa_x}x$ with $A\cap(y,x)=\emptyset$. Clearly this is a nontrivial ideal since $P_{\kappa_x}x\notin I_{\kappa_x,x}$.

Now suppose $\xi>0$. Let us show that $\NS_{\kappa_x,x}^\xi$ is an ideal. Suppose $A$ and $B$ are both not $\xi$-s-strongly stationary in $P_{\kappa_x}x$. By Corollary \ref{corollary_filter_base}, there are sets $C_A,C_B\subseteq P_{\kappa_x}x$ such that $C_A$ is $\zeta_A$-s-weak club in $P_{\kappa_x}x$ for some $\zeta_A<\xi$, $C_B$ is $\zeta_B$-s-weak club in $P_{\kappa_x}x$ for some $\zeta_B<\xi$, such that $C_A\cap A=\emptyset$ and $C_B\cap B=\emptyset$. Then $d_{\zeta_A}(C_A)\cap d_{\zeta_B}(C_B) \cap (A\cup B)=\emptyset$ where $d_{\zeta_A}(C_A)\cap d_{\zeta_B}(C_B)$ is $\zeta$-s-weak club in $P_{\kappa_x}x$ for $\zeta=\max\{\zeta_A,\zeta_B\}$ because $d_{\zeta_A}(C_A)\cap d_{\zeta_B}(C_B)\cap P_{\kappa_x}x$ is $\zeta$-s-strongly stationary in $P_{\kappa_x}x$ by Theorem \ref{theorem_induction}(3) and furthermore
\[d_\zeta(d_{\zeta_A}(C_A)\cap d_{\zeta_B}(C_B))\subseteq d_\zeta(C_A)\cap d_\zeta(C_B).\]
\end{proof}

\begin{theorem} \label{theorem_base_characterization}
    Suppose that $0 < \xi < \kappa$. Then the following are equivalent:
    \begin{enumerate}
        \item $\B_\xi$ is a base for $\tau_\xi$;
        \item for every $\zeta \leq \xi$, every $x \in P_\kappa X$, and 
        every $A \subseteq P_\kappa X$, if $A$ is $\zeta$-strongly 
        stationary in $P_{\kappa_x} x$, then $A$ is $\zeta$-s-strongly 
        stationary in $P_{\kappa_x} x$.
    \end{enumerate}
\end{theorem}

\begin{proof}
    For the forward direction, suppose that (2) fails, and let 
    $\zeta$, $x$, and $A$ form a counterexample, with $\zeta$ minimal 
    among all such counterexamples. Note that we must have 
    $\zeta > 0$.

    \begin{claim}
        $P_{\kappa_x} x$ is not $\zeta$-s-strongly stationary.
    \end{claim}

    \begin{proof}
        Suppose otherwise. We will show that $A$ is in fact 
        $\zeta$-s-strongly stationary, contradicting our choice 
        of $A$. By Corollary \ref{corollary_filter_base}, it 
        suffices to show that, for all $\eta < \zeta$ and every 
        $\eta$-s-weak club $C$ in $P_{\kappa_x} x$, we have 
        $A \cap C \neq \emptyset$. Fix such $\eta$ and $C$.
        Then $C$ is $\eta$-s-strongly stationary in $P_{\kappa_x} x$ 
        and hence, by the minimality of $\zeta$, $\eta$-strongly 
        stationary in $P_{\kappa_x} x$. Thus, since $A$ is $\zeta$-strongly stationary, there is $y \in A$ such that $A$ is 
        $\eta$-strongly stationary in $P_{\kappa_y} y$ and hence, 
        again by the minimality of $\zeta$, $\eta$-s-strongly stationary 
        in $P_{\kappa_y} y$. But then, since $C$ is an $\eta$-s-weak club 
        in $P_{\kappa_x} x$, we have $y \in C \cap A$, as desired.
    \end{proof}

    We can therefore fix an $\eta < \zeta$ and sets $S,T \subseteq 
    P_{\kappa_x} x$ such that $S$ and $T$ are both $\eta$-s-strongly 
    stationary in $P_{\kappa_x} x$ but there is no $y \in P_{\kappa_x} x$ 
    such that $S$ and $T$ are both $\eta$-s-strongly stationary in 
    $P_{\kappa_x} x$. Then we have $d_\eta(S) \cap d_\eta(T) = \{x\}$, 
    and hence $\{x\} \in \tau_\xi$. To show that (1) fails, it thus 
    suffices to show that $\{x\} \notin \B_\xi$.

    Since $P_{\kappa_x} x$ is $1$-strongly stationary, it follows that 
    $\kappa_x$ is a limit cardinal, and hence $\{x\} \notin \B_0$. 
    Now suppose that $B \subseteq P_{\kappa_x} x$, $\xi_0 < \xi$, and 
    $x \in d_{\xi_0}(B)$. Since $P_{\kappa_x} x$ is not $\zeta$-s-strongly 
    stationary, it follows that $\xi_0 < \zeta$ and $B$ is $\xi_0$-s-stationary 
    in $P_{\kappa_x} x$. By minimality of $\zeta$, $B$ is $\xi_0$-stationary 
    in $P_{\kappa_x} x$, so, since $P_{\kappa_x} x$ is $\zeta$-strongly 
    stationary, there is $y \in P_{\kappa_x} x$ such that $B$ is 
    $\xi_0$-strongly stationary in $P_{\kappa_y} y$. Again by minimality of 
    $\zeta$, $B$ is $\xi_0$-s-strongly stationary in $P_{\kappa_y} y$, 
    so $y \in d_{\xi_0}(B)$. It follows that $\{x\} \notin \B_\xi$.

    For the backward direction, suppose that (2) holds, and fix 
    $x \in P_\kappa X$, $I \in \B_0$, $0 < n < \omega$, ordinals 
    $\xi_0, \ldots, \xi_{n-1} < \xi$, and sets $A_0, \ldots, A_{n-1} 
    \subseteq P_{\kappa_x} x$ such that
    \[
      x \in I \cap d_{\xi_0}(A_0) \cap \ldots \cap d_{\xi_{n-1}}(A_{n-1}).
    \]
    Let $\zeta := \max\{\xi_0, \ldots, \xi_{n-1}\} < \xi$. It follows 
    that $P_{\kappa_x}$ is $\zeta$-s-strongly stationary. If 
    $P_{\kappa_x} x$ is not $(\zeta + 1)$-strongly stationary, then 
    there is $A \subseteq P_{\kappa_x} x$ such that $d_\zeta(A) = \{x\}$. 
    We can therefore assume that $P_{\kappa_x} x$ is $(\zeta + 1)$-strongly 
    stationary and hence, by (2), $(\zeta+1)$-s-strongly stationary. 
    But then it follows that $I \cap d_{\xi_0}(A_0) \cap \ldots \cap 
    d_{\xi_{n-1}}(A_{n-1})$ is $\zeta$-s-weak club in $P_{\kappa_x} x$. 
    In particular, it is $\zeta$-s-strongly stationary in $P_{\kappa_x} x$, 
    so 
    \[
      x \in d_\zeta(I \cap d_{\xi_0}(A_0) \cap \ldots \cap d_{\xi_{n-1}}(A_{n-1})) 
      \subseteq I \cap d_{\xi_0}(A_0) \cap \ldots \cap d_{\xi_{n-1}}(A_{n-1}),
    \]
    and $d_\zeta(I \cap d_{\xi_0}(A_0) \cap \ldots \cap d_{\xi_{n-1}}(A_{n-1})) 
    \in \B_\xi$. Therefore, $\B_\xi$ is a base for $\tau_\xi$.
\end{proof}

\subsection{Consequences of $\Pi^1_\xi$-indescribability}\label{section_indescribability}

In this section we establish the consistency of the $\xi$-s-strong stationarity of $P_{\kappa_x}x$, for $\xi \leq \kappa_x$, using a two-cardinal version of transfinite indescribability. 

The classical notion of $\Pi^m_n$-indescribability studied by Levy \cite{MR0281606} was generalized to the two-cardinal setting in a set of handwritten notes by Baumgartner (see \cite[Section 4]{MR0808767}). More recently, various transfinite generalizations of classical $\Pi^1_n$-indescribability, involving certain infinitary formulas have been studied in the cardinal context \cite{MR3894041, MR4094556, MR3416912, CodyHigherIndescribability, MR4206111, MR4594301} and in the two-cardinal context \cite{MR4082998}.

Let us review the definition of $\Pi^1_\xi$-indescribability in the two-cardinal context used in \cite{MR4082998}. For the reader's convenience, we review the notion of $\Pi^1_\xi$ formula introduced in \cite{MR3894041}. Recall that a formula of second-order logic is $\Pi^1_0$, or equivalently $\Sigma^1_0$, if it does not have any second-order quantifiers, but it may have finitely-many first-order quantifiers and finitely-many first and second-order free variables. Bagaria inductively defined the notion of $\Pi^1_\xi$ formula for any ordinal $\xi$ as follows. A formula is $\Sigma^1_{\xi+1}$ if it is of the form
\[\exists X_0\cdots\exists X_k\varphi(X_0,\ldots,X_k)\]
where $\varphi$ is $\Pi^1_\xi$, and a formula is $\Pi^1_{\xi+1}$ if it is of the form 
\[\forall X_0\cdots\forall X_k\varphi(X_0,\ldots, X_k)\]
where $\varphi$ is $\Sigma^1_\xi$. If $\xi$ is a limit ordinal, we say that a formula is $\Pi^1_\xi$ if it is of the form 
\[\bigwedge_{\zeta<\xi}\varphi_\zeta\]
where $\varphi_\zeta$ is $\Pi^1_\zeta$ for all $\zeta<\xi$ and the infinite conjunction has only finitely-many free second-order variables. We say that a formula is $\Sigma^1_\xi$ if it is of the form
\[\bigvee_{\zeta<\xi}\varphi_\zeta\]
where $\varphi_\zeta$ is $\Sigma^1_\zeta$ for all $\zeta<\xi$ and the infinite disjunction has only finitely-many free second-order variables.

The two-cardinal definition of $\Pi^1_\xi$-indescribability below uses the following two-cardinal version of the usual $V_\alpha$-hierarchy below a fixed cardinal $\kappa$. Suppose $\kappa$ is an uncountable regular cardinal and $X$ is a set of ordinals with $|X|\geq\kappa$. For $\alpha\leq\kappa$ we define
\begin{align*}
V_0(\kappa,X)&=X,\\
V_{\alpha+1}(\kappa,X)&=P_\kappa(V_\alpha(\kappa,X))\cup V_\alpha(\kappa,X)\text{ and}\\
V_\alpha(\kappa,X)&=\bigcup_{\eta<\alpha} V_\alpha(\kappa,X)\text{ if $\alpha$ is a limit.}
\end{align*}
Clearly $V_\kappa\subseteq V_\kappa(\kappa,X)$ and if $X$ is transitive then so is $V_\alpha(\kappa,X)$ for $\alpha\leq\kappa$. Furthermore, both $P_\kappa X$ and $P_\kappa X\times P_\kappa X$ are subsets of $V_\kappa(\kappa,X)$. For more regarding the expressive power of $\Pi^1_\xi$ formulas over structures of the form $(V_\kappa(\kappa,X),\in,R_0,\ldots,R_{n-1})$, where $R_0,\ldots,R_{n-1}\subseteq V_\kappa(\kappa,X)$, one may consult \cite[Section 3]{MR1635559} or \cite{MR4082998}.

\begin{definition}[{\cite{MR4082998}}]\label{definition_indescribability}
For $\xi<\kappa$ we say that $S\subseteq P_\kappa X$ is \emph{$\Pi^1_\xi$-indescribable in $P_\kappa X$} if for any $R_0,\ldots,R_{n-1}\subseteq V_\kappa(\kappa,X)$ and any $\Pi^1_\xi$ sentence $\varphi$ such that 
\[(V_\kappa(\kappa,X),\in,R_0,\ldots,R_{n-1})\models\varphi,\] there is an $x\in S$ such that $x\cap\kappa=\kappa_x$ and
\[(V_{\kappa_x}(\kappa_x,x),\in,R_0\cap V_{\kappa_x}(\kappa_x,x),\ldots,R_{n-1}\cap V_{\kappa_x}(\kappa_x,x))\models\varphi.\]
The collection
\[\Pi^1_\xi(\kappa,X)=\{A\subseteq P_\kappa X\st \textrm{$A$ is not $\Pi^1_\xi$-indescribable in $P_\kappa X$}\}\]
is called the \emph{$\Pi^1_\xi$-indescribability ideal on $P_\kappa X$}.
\end{definition}

Standard arguments, which we omit, establish the consistency of two-cardinal indescribability from supercompactness.

\begin{proposition} \label{proposition_supercompact}
Suppose $\kappa$ is $\lambda$-supercompact where $\kappa\leq\lambda$ and $\lambda^{<\kappa}=\lambda$. Then $P_\kappa\lambda$ is $\Pi^1_\xi$-indescribable for all $\xi<\kappa$. Furthermore, the set 
\[\{x\in P_\kappa\lambda\st \textrm{$\kappa_x=x\cap\kappa$ and $P_{\kappa_x}x$ is $\Pi^1_\xi$-indescribable for all $\xi<\kappa_x$}\}\]
is in any normal measure $U$ on $P_\kappa\lambda$.
\end{proposition}

Abe \cite[Lemma 4.1]{MR1635559} showed that if $P_\kappa X$ is $\Pi^1_n$-indescribable then $\Pi^1_n(\kappa,X)$ is a strongly normal ideal on $P_\kappa X$. As pointed out in \cite{CodyWhite}, a straightforward application of the arguments for \cite[Lemma 4.1]{MR1635559} and \cite[Proposition 4.4]{MR3894041}, which is left to the reader, establishes the following.

\begin{proposition} \label{proposition_strongly_normal_ind}
For $\xi<\kappa$, if $P_\kappa X$ is $\Pi^1_\xi$-indescribable then $\Pi^1_\xi(\kappa,X)$ is a strongly normal ideal on $P_\kappa X$.
\end{proposition}

Next we show that the $\xi$-s-strong stationarity of a set $S$ in $P_{\kappa_x}x$ can be expressed by a $\Pi^1_\xi$ formula.

\begin{lemma}\label{lemma_express}
For all $\xi<\kappa$ there is a $\Pi^1_\xi$ formula $\Phi_\xi(R,S,T)$ with three free second-order variables such that for $x\in P_\kappa X$, a set $A\subseteq P_{\kappa_x}x$ is $\xi$-s-strongly stationary in $P_{\kappa_x}x$ if and only if 
\[(V_{\kappa_x}(\kappa_x,x),\in,A,P_{\kappa_x}x,\strsub_x)\models\Phi_\xi[A,P_{\kappa_x}x,\strsub_x],\]
where $\strsub_x$ denotes the usual strong subset ordering $\strsub$ restricted to $P_{\kappa_x}x$.
\end{lemma}

\begin{proof}
We proceed by induction on $\xi$. We let $\Phi_0(R,S,T)$ be the $\Pi^1_0$ formula
\[(\forall y\in S) (\exists x\in R)\ (y,x)\in T\]
so that $\Phi_0[A,P_{\kappa_x}x,\strsub]$ expresses that $A$ is $0$-s-strongly stationary (i.e.\ $\strsub$-cofinal) in $P_{\kappa_x}x$ over the structure $(V_{\kappa_x}(\kappa_x,x),\in,A,P_{\kappa_x}x,\strsub)$. 

Suppose $\xi$ is a limit ordinal. It is easy to see that $\Phi_\xi = \bigwedge_{\zeta<\xi}\Phi_\zeta$ is as desired.

Suppose $\xi=\zeta+1$. Let $\Phi_\zeta$ be the $\Pi^1_\zeta$-formula obtained from the induction hypothesis. Then for all $x\in P_\kappa X$ a set $A\subseteq P_{\kappa_x}x$ is $\zeta$-s-strongly stationary in $P_{\kappa_x}x$ if and only if 
\[(V_{\kappa_x}(\kappa_x,x),\in,A,P_{\kappa_x}x,\strsub_x)\models\Phi_\zeta[A,P_{\kappa_x}x,\strsub_x].\]
For $x\in P_\kappa X$ we see that $A\subseteq P_{\kappa_x}x$ is $\xi$-s-strongly stationary in $P_{\kappa_x}x$ if and only if
\[\Phi_\zeta[A,P_{\kappa_x}x,\strsub_x] \land (\forall S\subseteq P_{\kappa_x}x)(\forall T\subseteq P_{\kappa_x}x)[\Phi_\zeta[S,P_{\kappa_x}x,\strsub_x]\land\Phi_\zeta[T,P_{\kappa_x}x,\strsub_x]\longrightarrow\]
\[(\exists y\in A) \Phi_\zeta[S\cap P_{\kappa_y}y,P_{\kappa_y}y,\strsub_y]\land \Phi_\zeta[T\cap P_{\kappa_y}y,P_{\kappa_y}y,\strsub_y]]\]
holds in $(V_{\kappa_x}(\kappa_x,x),\in,A,P_{\kappa_x}x,\strsub_x)$.
It is easy to check that the previous formula is equivalent to a $\Pi^1_\xi$ formula, hence the desired formula $\Phi_\xi(R,S,T)$ exists.
\end{proof}

\begin{corollary}\label{corollary_stationarity_from_indescribability}
For $x\in P_\kappa X$ with $\kappa_x=x\cap\kappa$, if $A\subseteq P_{\kappa_x}x$ is $\Pi^1_\xi$-indescribable in $P_{\kappa_x}x$ then $A$ is $\xi+1$-s-strongly stationary in $P_{\kappa_x}x$.
\end{corollary}

\begin{proof}
To show that $A$ is $\xi+1$-s-strongly stationary in $P_{\kappa_x}x$, fix sets $S,T\subseteq P_{\kappa_x}x$ that are $\zeta$-s-strongly stationary where $\zeta\leq \xi$ and let $\Phi_\zeta$ be the $\Pi^1_\zeta$-formula obtained from Lemma \ref{lemma_express}. Since $A$ is $\Pi^1_\xi$-indescribable, it is $\Pi^1_\zeta$-indescribable and the fact that 
\[(V_{\kappa_x}(\kappa_x,x),\in,S,T,P_{\kappa_x}x,\strsub_x)\models\Phi_\zeta[S,P_{\kappa_x}x,\strsub_x]\land\Phi_\zeta[T,P_{\kappa_x}x,\strsub_x]\]
implies that there is some $y\in A\cap P_{\kappa_x}x$ with $\kappa_y=y\cap\kappa$ such that the structure 
\[(V_{\kappa_y}(\kappa_y,y),\in,S\cap P_{\kappa_y}y,T\cap P_{\kappa_y}y,P_{\kappa_y}y,\strsub_y)\] 
satisfies
\[\Phi_\zeta[S\cap P_{\kappa_y}y,P_{\kappa_y}y,\strsub_y]\land\Phi_\zeta[T\cap P_{\kappa_y}y,P_{\kappa_y}y,\strsub_y],\]
and hence $S$ and $T$ are $\zeta$-s-strongly stationary in $y$. Therefore $A$ is $\xi+1$-s-strongly stationary in $P_{\kappa_x}x$.
\end{proof}

\begin{corollary}
For $\xi<\kappa$, if there is an $x\in P_\kappa X$ such that $P_{\kappa_x}x$ is $\Pi^1_\xi$-indescribable then the $\tau_{\xi+1}$-topology on $P_\kappa X$ is not discrete.
\end{corollary}

\begin{proposition}\label{proposition_indescribable_reflection}
Suppose $P_\kappa X$ is $\Pi^1_1$-indescribable. Then a set $A\subseteq P_\kappa X$ is $2$-s-strongly stationary in $P_\kappa X$ if and only if for every pair $S,T$ of strongly stationary subsets of $P_\kappa X$ there is an $x\in A$ such that $x\cap\kappa=\kappa_x$ is a Mahlo cardinal and the sets $S$ and $T$ are both strongly stationary in $P_{\kappa_x}x$.
\end{proposition}

\begin{proof}
Suppose $A$ is $2$-s-strongly stationary in $P_\kappa X$. Fix sets $S$ and $T$ that are strongly stationary in $P_\kappa X$. The fact that $\kappa$ is Mahlo and the sets $S$ and $T$ are strongly stationary in $P_\kappa X$ can be expressed by a $\Pi^1_1$ sentence:
\[(V_\kappa(\kappa,X),\in,P_\kappa X,S,T)\models\varphi.\]
The set
\[C=\{x\in P_\kappa X\st (V_{\kappa_x}(\kappa_x,x),\in,P_{\kappa_x}x,S\cap V_{\kappa_x}(\kappa_x,x),T\cap V_{\kappa_x}(\kappa_x,x))\models\varphi\}\]
is in the filter $\Pi^1_1(\kappa,X)^*$. Thus $C$ is, in particular, strongly stationary in $P_\kappa X$ and so by Lemma \ref{proposition_strong_vs_1_s_stationarity} we see that $C$ is $1$-s-strongly stationary in $P_\kappa X$. Since $A$ is $2$-s-strongly stationary in $P_\kappa X$, there is an $x\in A\cap C$ and it follows that $\kappa_x$ is Mahlo and the sets $S$ and $T$ are strongly stationary in $P_{\kappa_x}x$.

Conversely, to show that $A$ is $2$-s-strongly stationary in $P_\kappa X$, fix sets $Q$ and $R$ that are $1$-s-strongly stationary in $P_\kappa X$. By Lemma \ref{proposition_strong_vs_1_s_stationarity}, $Q$ and $R$ are strongly stationary in $P_\kappa X$. Thus, by assumption, there is an $x\in A$ such that $x\cap\kappa=\kappa_x$ is Mahlo and the sets $Q$ and $R$ are both strongly stationary in $P_{\kappa_x}x$. By Lemma \ref{proposition_strong_vs_1_s_stationarity}, $Q$ and $R$ are both $1$-s-strongly stationary in $P_{\kappa_x}x$. Hence $A$ is $2$-s-strongly stationary in $P_\kappa X$.
\end{proof}

\begin{proposition}\label{proposition_strongly_normal}
For $x\in P_\kappa X$ with $x\cap\kappa=\kappa_x$,
if $P_{\kappa_x}x$ is $\Pi^1_\xi$-indescribable where $\xi<\kappa_x$, then the ideal $\NS_{\kappa_x,x}^{\xi+1}$ (see Definition \ref{definition_xi_s_ideal}) is strongly normal.
\end{proposition}

\begin{proof}
Suppose $C_z\in (\NS^{\xi+1}_{\kappa_x,x})^*$ for $z\in P_{\kappa_x}X$. Without loss of generality, by Corollary \ref{corollary_xi_plus_1_s_filter_base}, we may assume that each $C_z$ is $\xi$-s-weak club in $P_{\kappa_x}x$.

Since each $C_z$ is in the filter $\Pi^1_\xi(\kappa_x,x)^*$ and $\Pi^1_\xi(\kappa_x,x)$ is strongly normal, it follows that the set $C=\bigtriangleup_\strsub\{C_z\st z\in P_{\kappa_x}x\}$ is in the filter $\Pi^1_\xi(\kappa_x,x)^*$ and thus $C$ is $\xi+1$-s-strongly stationary in $P_{\kappa_x}x$ by Corollary \ref{corollary_stationarity_from_indescribability}. By Theorem \ref{theorem_induction}(2), it follows that $d_\xi(C)$ is $\xi$-s-strongly stationary in $P_{\kappa_x}x$, and since $d_\xi$ is the Cantor derivative of the space $(P_\kappa X,\tau_\xi)$, it follows that $d_\xi(d_\xi(C))\subseteq d_\xi(C)$ and hence $d_\xi(C)$ is $\xi$-s-weak club in $P_\kappa X$. Thus it will suffice to show that $d_\xi(C)\subseteq C$.

Let us verify that $d_\xi(C)\subseteq \bigtriangleup_\strsub\{d_\xi(C_z)\st z\in P_{\kappa_x}x\}$. Suppose $y\in d_\xi(C)$, then the set $\bigtriangleup_\strsub\{C_z\st z\in P_{\kappa_x}x\}$ is $\xi$-s-strongly stationary in $P_{\kappa_y}y$. To show that $y\in \bigtriangleup_\strsub\{d_\xi(C_z)\st z\in P_{\kappa_x}x\}$ we must verify that $y\in\bigcap_{z\strsub y}d_\xi(C_z)$. Fix $z\strsub y$, then $(z,y)\cap \bigtriangleup_\strsub\{C_z\st z\in P_{\kappa_x}x\}\subseteq C_z$ and since $(z,y)\cap \bigtriangleup_\strsub\{C_z\st z\in P_{\kappa_x}x\}$ is $\xi$-s-strongly stationary in $P_{\kappa_y}y$ we see that $y\in d_\xi(C_z)$. Thus $d_\xi(C)\subseteq \bigtriangleup_\strsub\{d_\xi(C_z)\st z\in P_{\kappa_x}x\}$.

Since each $C_z$ is $\xi$-s-weak club in $P_{\kappa_x}x$, it follows that $d_\xi(C_z)\subseteq C_z$ and thus
\[d_\xi(C)\subseteq\bigtriangleup_\strsub\{d_\xi(C_z)\st z\in P_{\kappa_x}x\}\subseteq \bigtriangleup_\strsub\{C_z\st z\in P_{\kappa_x}x\}=C.\]
\end{proof}

\subsection{Variations}\label{section_variation}\label{section_variations}

In this subsection, we investigate a couple of variations on the 
sequence of derived topologies considered above. First, we show that 
by restricting our attention to a certain natural club subset of 
$P_\kappa X$, certain questions about the resulting spaces become 
more tractable.

Let $P'_\kappa X$ be the set of $x \in P_\kappa X$ for which 
$\kappa_x = x \cap \kappa$. Similarly, if $x \in P'_\kappa X$, 
then $P'_{\kappa_x} x = P'_\kappa X \cap P_{\kappa_x} x$. 
If $\kappa$ is weakly inaccessible, then 
$P'_\kappa X$ is evidently a club, and hence a weak club, in $P_\kappa X$.
It follows that, if $\xi < \kappa$, $x \in P_\kappa X$, and $\kappa_x$ is weakly 
inaccessible, then 
\begin{align} \label{variation_eq}
  (P_{\kappa_x} x \text{ is $\xi$-s-stationary}) \Longleftrightarrow 
  (P'_{\kappa_x} x \text{ is $\xi$-s-stationary in } P_{\kappa_x} x).
\end{align}

For each $\xi < \kappa$, let $\tau'_\xi$ be the subspace topology 
on $P'_\kappa X$ induced by $\tau_\xi$, and let $\B'_\xi = 
\{U \cap P'_\kappa X \mid U \in \B_\xi\}$; it follows that $\tau'_\xi$ 
is the topology on $P'_\kappa X$ generated by $\B'_\xi$. 

\begin{proposition} \label{isolated_prop}
    Suppose that $x \in P'_\kappa X$. Then the following are equivalent:
    \begin{enumerate}
        \item $\kappa_x$ is weakly inaccessible;
        \item $x$ is not isolated in $(P'_\kappa X, \tau'_0)$.
    \end{enumerate}
\end{proposition}

\begin{proof}
    If $\kappa_x$ is weakly inaccessible and $y \prec x$, with 
    $y \in P_\kappa X$, then, letting $\lambda$ be the least cardinal 
    with $|y| < \lambda$, we have $y \cup \lambda \in (y,x] \cap P'_\kappa 
    X$. The implication (1)$\implies$(2) follows immediately.

    For the converse, suppose first that $\kappa_x = \lambda^+$ is a 
    successor cardinal, and let $y \prec x$ be such that $|y| = \lambda$. 
    Then $(y,x] = \{x\}$, so $x$ is isolated in $\tau_0$, and hence also 
    in $\tau'_0$. Suppose next that $\kappa_x$ is singular, and let 
    $y \subseteq \kappa_x$ be a cofinal subset such that $|y| = 
    \cf(\kappa_x)$. Then $(y,x] \cap P'_\kappa X = \{x\}$, so $x$ 
    is isolated in $\tau'_0$.
\end{proof}

Using this proposition, we can establish the following characterization of 
when $\B'_\xi$ forms a base for $\tau'_\xi$. Since the proof is essentially 
the same as that of Theorem \ref{theorem_base_characterization}, we leave 
it to the reader.

\begin{theorem} \label{theorem_base_characterization_ii}
    Suppose that $0 < \xi < \kappa$. Then the following are equivalent:
    \begin{enumerate}
        \item $\B'_\xi$ is a base for $\tau'_\xi$;
        \item for every $\zeta \leq \xi$, every $x \in P'_\kappa X$ for which 
        $\kappa_x$ is weakly inaccessible, and every $A \subseteq P_\kappa X$, 
        if $A$ is $\zeta$-strongly stationary in $P_{\kappa_x} x$, then $A$ 
        is $\zeta$-s-strongly stationary in $P_{\kappa_x} x$. \qed
    \end{enumerate}
\end{theorem}

\begin{corollary}
    $\B'_1$ is a base for $\tau'_1$.
\end{corollary}

\begin{proof}
    This is immediate from Proposition \ref{proposition_strong_vs_1_s_stationarity}
    and Theorem \ref{theorem_base_characterization_ii}.
\end{proof}

We saw above that the topology $(P_\kappa X, \tau_1)$ can be 
characterized by specifying that, if $x \in P_\kappa X$ and 
$A \subseteq P_\kappa X$, then $x$ is a limit point of $A$ if 
and only if $A$ is strongly $1$-s-stationary in $P_{\kappa_x} x$. 
By Proposition \ref{proposition_strong_vs_1_s_stationarity},
if $\kappa_x$ is regular, then this is equivalent to $A$ being 
$1$-strongly stationary in $P_{\kappa_x} x$, and if $\kappa_x$ is 
Mahlo, it is in turn equivalent to $A$ being strongly stationary 
in $P_{\kappa_x} x$. One can ask if there is a variant on this 
topology in which limit points are characterized by 
\emph{stationarity} in the sense of \cite{MR0325397} (recall 
the discussion at the end of Section \ref{section_strong_stationarity}). 
We now show that the answer is positive as long as $\kappa$ is 
weakly inaccessible and one only 
requires this of $x \in P_\kappa X$ for which $\kappa_x$ is weakly 
inaccessible. We first establish the following proposition.

\begin{proposition} \label{proposition_stat}
    Suppose that $\kappa$ is weakly inaccessible, $A \subseteq P_\kappa X$ and the set
    \[
      \{x \in P_\kappa X \mid \kappa_x \text{ is regular and }
      A \cap P_{\kappa_x} x \text{ is stationary in }
      P_{\kappa_x}x\}
    \]
    is stationary in $P_\kappa X$. Then $A$ is stationary.
\end{proposition}

\begin{proof}
    Fix a club $C$ in $P_\kappa X$. Since $\kappa$ is regular and 
    uncountable, by \cite[Theorem 1.5]{MR0357121}, we can find a function $f : [X]^2 \rightarrow P_\kappa X$ 
    such that $B_f \subseteq C$, where 
    \[
      B_f := \{x \in P_\kappa X \mid f``[x]^2 \subseteq P(x)\}.
    \]
    We actually get slightly more. Namely, let $C_f$ be the 
    set of $x \in P_\kappa X$ for which $f``[x]^2 \subseteq 
    P_{\kappa_x}x$. Then clearly $C_f \subseteq B_f \subseteq C$, and, 
    moreover, $C_f$ is club in $P_\kappa X$. To see this, simply note 
    that $C_f$ is clearly closed and, if $\langle y_n \mid n < \omega 
    \rangle$ is a $\prec$-increasing sequence of elements of 
    $B_f$, then $\bigcup \{y_n \mid n < \omega\} \in C_f$, so 
    $C_f$ is cofinal in $P_\kappa X$. (This is where we use the 
    fact that $\kappa$ is inaccessible, and hence a limit cardinal).

    By assumption, we can find $x \in P_\kappa X$ such that
    \begin{enumerate}
        \item $\kappa_x$ is regular;
        \item $A \cap P_{\kappa_x} x$ is stationary in $P_{\kappa_x} x$;
        \item $x \in C_f$.
    \end{enumerate}
    Since $x \in C_f$, we know that $g := f \restriction [x]^2$ satisfies 
    $g:[x]^2 \rightarrow P_{\kappa_x}x$. Since $\kappa_x$ is regular, 
    it follows that $B_g$ is a club in $P_{\kappa_x}x$. Then item (2) 
    above implies that $B_g \cap A \cap P_{\kappa_x} x \neq 0$. 
    Since $B_g \subseteq B_f \subseteq C$, it follows that $C \cap A \neq 0$. 
    The choice of $C$ was arbitrary, and hence $A$ is stationary in 
    $P_\kappa X$.
\end{proof}

Note that Proposition \ref{proposition_stat} fails if $\kappa > \aleph_1$ is a
successor cardinal. Indeed, if $\kappa = \nu^+ > \aleph_1$, then 
$A = P_\nu X$ satisfies the hypothesis of 
Proposition \ref{proposition_stat} but is not stationary in 
$P_\kappa X$.

Now, if $\kappa$ is weakly inaccessible, define a function 
$c:P(P_\kappa X) \rightarrow P(P_\kappa X)$ by letting 
\[
  c(A) = A \cup \{x \in P_\kappa X \mid \kappa_x \text{ is weakly 
  inaccessible and } A \cap P_{\kappa_x} x \text{ is stationary in }
  P_{\kappa_x} x\}.
\]
Proposition \ref{proposition_stat} implies that $c$ is a closure 
operator. If $\tau$ is the topology 
\[\{U\subseteq P_\kappa X\st c(P_\kappa X\setminus U)=P_\kappa X\setminus U\}\]
on $P_\kappa X$ generated by $c$, then, clearly $\tau$ is a witness to the following.

\begin{corollary}\label{corollary_stationary}
If $\kappa$ is weakly inaccessible and $X$ is a set of ordinals with $\kappa\subseteq X$, then there is a topology $\tau$ on $P_\kappa X$ such that for $A\subseteq P_\kappa X$, $x$ is a limit point of $A$ if and only if $\kappa_x$ is weakly inaccessible and $A\cap P_{\kappa_x}x$ is stationary in $P_{\kappa_x}x$. In particular, $x$ is a nonisolated point of the space $(P_\kappa X,\tau)$ if and only if $\kappa_x$ is weakly inaccessible.
\end{corollary}

\section{On Ramseyness and indescribability}\label{section_ramsey}

In this section we answer questions concerning the relationship between Ramseyness and indescribability, which were raised by the first author and Peter Holy \cite{MR4594301} in the context of cardinals, and by the first author and Philip White \cite{CodyWhite} in the two-cardinal context. We provide detailed arguments in the cardinal context and simply state definitions and results in the two-cardinal context since the proofs are similar.

Let us review the definition and some basic properties of canonical functions. We follow the definitions and notation given in \cite{MR2768692}. The sequence of canonical functions $\<f_\alpha\st\alpha<\lambda^+\>$ is a sequence of canonical representatives of the ordinals less than $\lambda^+$ in the generic ultrapower obtained by forcing with any normal ideal $I$ on $Z\subseteq P(\lambda)$. We recursively define $\<f_\alpha\st\alpha<\lambda^+\>$ as follows. For $\alpha<\lambda$ we let 
\[f_\alpha(z)=\ot(z\cap\alpha)\]
for all $z\in Z$. For $\lambda<\alpha<\lambda^+$ define
\[f_\alpha(z)=\sup\{f_{b_{\lambda,\alpha}(\eta)}(z)+1\st \eta\in z\}\]
where $b_{\lambda,\alpha}:\lambda\to\alpha$ is a bijection. Let us note that if we take $Z=\lambda$, then each $f_\alpha$ represents the ordinal $\alpha$ in any generic ultrapower obtained by forcing with a normal ideal on $\lambda$. Whereas, in the two-cardinal setting, if we take $Z=P_\kappa\lambda$, the function $f_\alpha$ represents $\alpha$ in any generic ultrapower obtained by forcing with a normal ideal on $P_\kappa\lambda$.



Let us review some basic definitions concerning ineffable and Ramsey operators on cardinals. For $S\subseteq\kappa$, we say that $\vec{S}=\<S_\alpha\st\alpha\in S\>$ is an \emph{$S$-list} if $S_\alpha\subseteq\alpha$ for all $\alpha\in S$. Given an $S$-list $\vec{S}$, a set $H\subseteq S$ is said to be \emph{homogeneous for $\vec{S}$} if whenever $\alpha,\beta\in H$ with $\alpha<\beta$ we have $S_\alpha=S_\beta\cap\alpha$. If $I$ is an ideal on $\kappa$, we define another ideal $\I(I)$ such that for $S\subseteq\kappa$ we have $S\in\I(I)^+$ if and only if for every $S$-list $\vec{S}=\<S_\alpha\st\alpha\in S\>$ there is a set $H\in P(S)\cap I^+$ which is homogeneous for $\vec{S}$. We say that $\kappa$ is \emph{almost ineffable} if $\kappa\in\I([\kappa]^{<\kappa})^+$ and $\kappa$ is \emph{ineffable} if $\kappa\in\I(\NS_\kappa)^+$. The function $\mathcal{I}$ is referred to as the \emph{ineffable operator on $\kappa$}.

Recall that for a cardinal $\kappa$ and a set $S\subseteq\kappa$, a function $f:[\kappa]^{<\omega}\to\kappa$ is called \emph{regressive on $S$} if $f(x)<\min(x)$ for all $x\in [S]^{<\omega}$. Given a function $f:[\kappa]^{<\omega}\to \kappa$, a set $H\subseteq\kappa$ is said to be \emph{homogeneous for $f$} if $f\restrict [H]^n$ is constant for every $n<\omega$. If $I$ is an ideal on a cardinal $\kappa$, we define another ideal $\R(I)$ such that for $S\subseteq\kappa$ we have $S\in\R(I)^+$ if and only if for every function $f:[\kappa]^{<\omega}\to\kappa$ that is regressive on $S$, there is a set $H\in P(S)\cap I^+$ which is homogeneous for $f$. We say that a set $S\subseteq\kappa$ is \emph{Ramsey in $\kappa$} if $S\in\R([\kappa]^{<\kappa})^+$. Let us note that the definition of Ramsey set and, more generally, the definition of $\R(I)$ given above are standard and have many equivalent formulations (see \cite[Proposition 2.8 and Theorem 2.10]{MR4206111} for details). The function $\mathcal{R}$ is called the \emph{Ramsey operator on $\kappa$}.

For a given ideal $I$ and ideal operator $\O$, such as $\O\in\{\I,\R\}$, we inductively define new ideals by letting
\begin{align*}
\O^0(I)&=I,\\
\O^{\alpha+1}(I)&=\O(\O^\alpha(I)) \textrm{ and}\\
\O^{\alpha}(I)&=\bigcup_{\beta<\alpha}\O^\beta(I).
\end{align*}
We say that a set $S\subseteq\kappa$ is \emph{$\gamma$-almost ineffable} if $S\in\I^\gamma([\kappa]^{<\kappa})$ and we say that $S\subseteq\kappa$ is \emph{$\gamma$-Ramsey in $\kappa$} if $S\in\R^\gamma([\kappa]^{<\kappa})^+$. So, for example, a set $S\subseteq\kappa$ is $1$-Ramsey in $\kappa$ if and only if it is Ramsey in $\kappa$, and $S$ is $2$-Ramsey in $\kappa$ if and only if for every function $f:[\kappa]^{<\omega}\to\kappa$ that is regressive on $S$ there is a set $H$ that is Ramsey in $\kappa$ and homogeneous for $f$. 

Recall that a set $S\subseteq\kappa$ is \emph{$\Pi^1_n$-indescribable in $\kappa$} if $(V_\kappa,\in,R)\models\varphi$ implies there is an $\alpha\in S$ with $(V_\alpha,\in,R\cap V_\alpha)\models\varphi$ whenever $S\subseteq V_\kappa$ and $\varphi$ is a $\Pi^1_n$ sentence. Recall that $\varphi$ is $\Pi^1_0$ if it is first order with one second-order free-variable. When $\kappa$ is $\Pi^1_n$-indescribable, the collection $\Pi^1_n(\kappa)$ of all subseteq of $\kappa$ which are not $\Pi^1_n$-indescribable in $\kappa$ forms a normal ideal on $\kappa$ \cite{MR0281606}. Baumgartner  studied ideals on $\kappa$ of the form $\I^\gamma(\Pi^1_n(\kappa))$ for $\gamma<\kappa^+$ and $n\in\omega\cup\{-1\}$ where for notational convenience we take $\Pi^1_{-1}(\kappa)=[\kappa]^{<\kappa}$ (see \cite[Section 7]{MR0384553} and \cite{MR0540770}). Ideals of the form $\R^\gamma([\kappa]^{<\kappa})$ and $\R^\gamma(\NS_\kappa)$ were introduced by Feng \cite{MR1077260}; note that if $\kappa$ is inaccessible then $\Pi^1_0(\kappa)=\NS_\kappa$.

Bagaria \cite{MR3894041} introduced a notion of the $\Pi^1_\xi$-indescribability of a cardinal $\kappa$ for $\xi<\kappa$. The first author \cite{CodyHigherIndescribability} extended Bagaria's definition and introduced a notion of $\Pi^1_\xi$-indescribability of $\kappa$ for $\xi<\kappa^+$.\footnote{Let us note that previously, Sharpe and Welch \cite{MR2817562} had used games to define a notion of $\Pi^1_\xi$-indescribability of a cardinal $\kappa$ for all $\kappa<\xi^+$, but the relationship between their notion and that of \cite{CodyHigherIndescribability} is still not known.} Instead of reviewing the rather lengthy definition, we refer the reader to \cite{CodyHigherIndescribability} for the definition of the $\Pi^1_\xi$-indescribability of a subset $S$ of $\kappa$ for $\xi<\kappa^+$. The $\Pi^1_\xi$-indescribability ideal on $\kappa$ is then 
\[\Pi^1_\xi(\kappa)=\{S\subseteq\kappa\st\text{$S$ is not $\Pi^1_\xi$-indescribable in $\kappa$}\}.\]
Let us note that, in some sense, the definition of $\Pi^1_\xi$-indescribability does not play a large role in what follows because it is being ``black boxed'' by Lemma \ref{lemma_ideal_containment} and Theorem \ref{theorem_almost_ineffable} (see the proof of Corollary \ref{corollary_indescribable_initial_segment}).

Ideals of the form $\R^\gamma(\Pi^1_\xi(\kappa))$ were studied by the first author \cite{MR4206111}, and more generally, ideals of the form $\I^\gamma(\Pi^1_\xi(\kappa))$ and $\R^\gamma(\Pi^1_\xi(\kappa))$ for $\gamma<\kappa^+$ and $\xi\in\kappa\cup\{-1\}$ were studied by the first author and Peter Holy \cite{MR4594301} (in fact the framework presented in \cite{MR4594301} handles many ideal operators other than $\I$ and $\R$).

Notice that for a carinal $\kappa$, to each ideal of the form $\O^\gamma(\Pi^1_\xi(\kappa))$ where $\O\in\{\I,\R\}$, $\gamma<\kappa^+$ and $\xi\in\kappa^+\cup\{-1\}$, there is a corresponding large cardinal hypothesis, namely $\kappa\in\O^\gamma(\Pi^1_\xi(\kappa))^+$.

\begin{definition}
Suppose $\kappa$ is a cardinal, $\gamma<\kappa^+$ and $\xi\in\kappa^+\cup\{-1\}$. Let $\Pi^1_{-1}(\kappa)=[\kappa]^{<\kappa}$. We say that $\kappa$ is \emph{$\gamma$-$\Pi^1_\xi$-ineffable} if $\kappa\in\I^\gamma(\Pi^1_\xi(\kappa))^+$, and $\kappa$ is \emph{$\gamma$-$\Pi^1_\xi$-Ramsey} if $\kappa\in\R^\gamma(\Pi^1_\xi(\kappa))^+$.
\end{definition}
So for example, $\kappa$ is $\gamma$-$\Pi^1_{-1}$-ineffable if and only if it is $\gamma$-almost ineffable, and $\kappa$ is $\gamma$-$\Pi^1_{-1}$-Ramsey if and only if it is $\gamma$-Ramsey.

Recall that Baumgartner proved \cite[Theorem 4.1]{MR0384553} that when $\kappa$ is a subtle cardinal the set 
\[\{\alpha<\kappa\st \textrm{$\alpha$ is $\Pi^1_n$-indescribable for all $n<\omega$}\}\]
is in the subtle filter on $\kappa$. More generally, the first author and Peter Holy proved that when $\kappa$ is subtle the set
\[\{\alpha<\kappa\st\textrm{$\alpha$ is $\Pi^1_\xi$-indescribable for all $\xi<\alpha^+$}\}\]
is in the subtle filter on $\kappa$. Since whenever a cardinal $\kappa$ is Ramsey it must also be subtle, it follows that the existence of a Ramsey cardinal is strictly stronger in consistency strength than the existence of a cardinal $\kappa$ such that $\kappa$ is $\Pi^1_\xi$-indescribable for all $\xi<\kappa^+$. Furthermore, as shown in \cite{MR4594301}, this result can be pushed up the almost ineffability hierarchy, in the sense that the existence of a $2$-almost ineffable cardinal is strictly stronger than the existence of a cardinal $\kappa$ that is $1$-$\Pi^1_\xi$-ineffable for all $\xi<\kappa^+$ as follows.

\begin{theorem}[{\cite[Theorem 3.8]{MR4594301}}]\label{theorem_almost_ineffable} Suppose $\gamma<\kappa^+$, $S\in\I^{\gamma+1}([\kappa]^{<\kappa})^+$ and $\vec{S}=\<S_\alpha\st\alpha\in S\>$ is an $S$-list. Let $A$ be the set of all ordinals $\alpha$ such that
\[\exists X\subseteq S\cap\alpha\left[(\forall\xi<\alpha^+\ X\in\I^{f^\kappa_\gamma(\alpha)}(\Pi^1_\xi(\alpha))^+)\land(X\cup\{\alpha\}\textrm{ is hom. for $\vec{S}$})\right].\]
Then, $S\setminus A\in\I^{\gamma+1}([\kappa]^{<\kappa})$.
\end{theorem}

\begin{corollary}[{\cite[Corollary 3.9]{MR4594301}}]\label{corollary_ineffable_hierarcy}
Suppose $\kappa\in\I^{\gamma+1}([\kappa]^{<\kappa})^+$ where $\gamma<\kappa^+$. Then the set
\[\{\alpha<\kappa\st(\forall\xi<\alpha^+)\ \alpha\in\I^{f^\kappa_\gamma(\alpha)}(\Pi^1_\xi(\alpha))^+\}\]
is in the filter $\I^{\gamma+1}([\kappa]^{<\kappa})^*$. In other words, if $\kappa$ is $\gamma+1$-almost ineffable then the set of $\alpha<\kappa$ which are $f^\kappa_\gamma(\alpha)$-$\Pi^1_\xi$-ineffable for all $\xi<\alpha$ is in the filter $\I^{\gamma+1}([\kappa]^{<\kappa})^*$.
\end{corollary}

\subsection{New results on Ramseyness and indescribability}

Now let us address the following question, and its generalizations, which were originally posed in \cite{MR4594301}.

\begin{question}\label{question_ramseyness}
Is the existence of a $2$-Ramsey cardinal strictly stronger than the existence of a cardinal $\kappa$ such that $\kappa\in \R(\Pi^1_\xi(\kappa))^+$ for all $\xi<\kappa^+$?
\end{question}

The following lemma is standard and is an easy consequence of Feng's characterization of Ramsey sets in terms of $(\omega,S)$-sequences \cite[Theorem 2.3]{MR1077260}.

\begin{lemma}\label{lemma_ideal_containment}
Suppose $\kappa$ is a Ramsey cardinal. Then
\[\I([\kappa]^{<\kappa})\subseteq\R([\kappa]^{<\kappa}).\]
\end{lemma}

\begin{corollary}[{\cite[Theorem 10.3]{MR4594301}}]\label{corollary_indescribable_initial_segment}
Suppose $S\in\R([\kappa]^{<\kappa})^+$ and let
\[T=\{\alpha\in S\st (\forall\xi<\alpha^+)\ S\cap\alpha\in \Pi^1_\xi(\alpha)^+\}.\]
Then $S\setminus T\in\R([\kappa]^{<\kappa})$.
\end{corollary}

\begin{proof}

Suppose $S\in\R([\kappa]^{<\kappa})^+$ and let $T$ be as in the statement of the corollary. By Lemma \ref{lemma_ideal_containment} we see that $S\in\I([\kappa]^{<\kappa})^+$ and by Theorem \ref{theorem_almost_ineffable} we have $S\setminus A\in\I([\kappa]^{<\kappa})\subseteq\R([\kappa]^{<\kappa})$. But $A\subseteq T$ so $S\setminus T\subseteq S\setminus A$ and hence $S\setminus T\in\R([\kappa]^{<\kappa})$.
\end{proof}

The next result shows that Corollary \ref{corollary_indescribable_initial_segment} can, in a sense, be pushed up the Ramsey hierarchy, and provides an affirmative answer to Question 10.4, Question 10.5, Question 10.6 and Question 10.9 in \cite{MR4594301}; it is at present the best known generalization of Theorem \ref{theorem_almost_ineffable} from the context of the ineffable operator to that of the Ramsey operator.

\begin{theorem}\label{theorem_ramsey}
Suppose $\gamma<\kappa^+$, $S\in\R^{\gamma+1}([\kappa]^{<\kappa})^+$ and let 
\[T=\{\alpha\in S\st (\forall \xi<\alpha^+)\ S\cap\alpha\in \R^{f^\kappa_\gamma(\alpha)}(\Pi^1_\xi(\alpha))^+\}.\] Then $S\setminus T\in\R^{\gamma+1}([\kappa]^{<\kappa})$.
\end{theorem}

\begin{proof}
If $\gamma=0$ the result follows directly from Corollary \ref{corollary_indescribable_initial_segment}. 

Suppose $\gamma=\delta+1<\kappa^+$ is a successor ordinal, and suppose for a contradiction that $S\setminus T\in\R^{\gamma+1}([\kappa]^{<\kappa})^+$. Recall that the set $C=\{\alpha<\kappa\st f^\kappa_{\delta+1}(\alpha)=f^\kappa_\delta(\alpha)+1\}$ is club in $\kappa$ and thus the set
\[E=(S\setminus T)\cap C\]
is in $\R^{\gamma+1}([\kappa]^{<\kappa})^+$. For each $\alpha\in E$ fix $\xi_\alpha<\alpha^+$ such that \[S\cap\alpha\in\R^{f^\kappa_{\delta}(\alpha)+1}(\Pi^1_{\xi_\alpha}(\alpha))=\R(\R^{f^\kappa_\delta(\alpha)}(\Pi^1_{\xi_\alpha}(\alpha)),\]
and fix a regressive function $g_\alpha:[S\cap\alpha]^{<\omega}\to\alpha$ which has no homogeneous set in $\R^{f^\kappa_\delta(\alpha)}(\Pi^1_{\xi_\alpha}(\alpha))^+$. Let $f:[E]^{<\omega}\to\kappa$ be a regressive function such that
\[f(\alpha_0,\ldots,\alpha_n)=g_{\alpha_n}(\alpha_0,\ldots,\alpha_{n-1})\]
for $n<\omega$ and $(\alpha_0,\ldots,\alpha_n)\in [E]^{n+1}$. Since $E\in\R^{\gamma+1}([\kappa]^{<\kappa})^+$, there is a set $H\in P(E)\cap\R^{\delta+1}([\kappa]^{<\kappa})^+$ homogeneous for $f$. By the inductive hypothesis it follows that if we let
\[T_H=\{\alpha\in H\st(\forall\xi<\alpha^+)\ H\cap\alpha\in\R^{f^\kappa_\delta(\alpha)}(\Pi^1_\xi(\alpha))^+\}\]
then $H\setminus T_H\in \R^{\delta+1}([\kappa]^{<\kappa})$. Thus we can fix an $\alpha\in T_H$. It follows that $H\cap\alpha\in\R^{f^\kappa_\delta(\alpha)}(\Pi^1_{\xi_\alpha}(\alpha))^+$ and $H\cap\alpha\subseteq E\cap\alpha\subseteq S\cap\alpha$ is homogeneous for $g_\alpha$, a contradiction.

Now suppose $\gamma<\kappa^+$ is a limit ordinal, and suppose again for contradiction that $S\setminus T\in\R^{\gamma+1}([\kappa]^{<\kappa})^+$. Recall that $C=\{\alpha<\kappa\st f^\kappa_\gamma(\alpha)\text{ is a limit ordinal}\}$ is a club subset of $\kappa$ and thus
\[E=(S\setminus T)\cap C\]
is in $R^{\gamma+1}([\kappa]^{<\kappa})^+$. For each $\alpha\in E$, using the fact that $\alpha\notin T$, let $\xi_\alpha<\alpha^+$ be such that \[S\cap\alpha\in \R^{f^\kappa_\gamma(\alpha)}(\Pi^1_{\xi_\alpha}(\alpha)).\]
Since $f^\kappa_\gamma(\alpha)=\sup\{f^\kappa_{b_{\kappa,\gamma}(\eta)}(\alpha)+1\st \eta\in \alpha\}<\alpha^+$ is a limit ordinal, we can choose an ordinal $r(\alpha)<\alpha$ such that 
\[S\cap\alpha\in \R^{f^\kappa_{b_{\kappa,\gamma}(r(\alpha))}(\alpha)+1}(\Pi^1_{\xi_\alpha}(\alpha)).\]
This defines a regressive function $r:E\to\kappa$, and by normality of $\R^{\gamma+1}([\kappa]^{<\kappa})$, there is an $E^*\in P(E)\cap \R^{\gamma+1}([\kappa]^{<\kappa})^+$ and some $\beta_0<\kappa$ such that $g(\alpha)=\beta_0$ for all $\alpha\in E^*$. Let $\nu=b_{\kappa,\gamma}(\beta_0)$ and notice that for all $\alpha\in E^*$,
\[S\cap\alpha\in\R^{f^\kappa_\nu(\alpha)+1}(\Pi^1_{\xi_\alpha}(\alpha)).\]

For each $\alpha\in E^*$, we fix a regressive function $g_\alpha:[S\cap\alpha]^{<\omega}\to\kappa$ which has no homogeneous set in $\R^{f^\kappa_\nu(\alpha)}(\Pi^1_{\xi_\alpha}(\alpha))^+$. Let $f:[E^*]^{<\omega}\to\kappa$ be a regressive function such that
\[f(\alpha_0,\ldots,\alpha_n)=g_{\alpha_n}(\alpha_0,\ldots,\alpha_{n-1})\]
for $n<\omega$ and $(\alpha_0,\ldots,\alpha_n)\in[E]^{n+1}$. Since $E^*\in \R^{\gamma+1}([\kappa]^{<\kappa})^+$ there is a set $H\in P(E^*)\cap \R^\gamma([\kappa]^{<\kappa})^+$ homogeneous for $f$. Since $\nu<\gamma$ we have 
\[H\in\R^\gamma([\kappa]^{<\kappa})^+\subseteq\R^{\nu+1}([\kappa]^{<\kappa})^+,\]
and we may apply the inductive hypothesis to see that the set
\[T_H=\{\alpha\in H\st(\forall\xi<\alpha^+)\ H\cap\alpha\in \R^{f^\kappa_\nu(\alpha)}(\Pi^1_\xi(\alpha))^+\}\]
satisfies $H\setminus T_H\in\R^{\nu+1}([\kappa]^{<\kappa})$. Thus we can fix an $\alpha\in T_H$. But then the set \[H\cap\alpha\subseteq E^*\cap\alpha\subseteq E\cap\alpha\subseteq S\cap\alpha\] is in $\R^{f^\kappa_\nu(\alpha)}(\Pi^1_{\xi_\alpha}(\alpha))^+$ and is homogeneous for $g_\alpha$, which is a contradiction.
\end{proof}

\begin{corollary}
Suppose $\gamma<\kappa^+$. If $\kappa$ is $\gamma+1$-Ramsey then the set of $\alpha<\kappa$ which are $f^\kappa_\gamma(\alpha)$-$\Pi^1_\xi$-Ramsey for all $\xi<\alpha^+$ is in the filter $\R^{\gamma+1}([\kappa]^{<\kappa})^*$.
\end{corollary}

\subsection{New results on two-cardinal Ramseyness}

Let us now discuss two-cardinal versions of the ineffable and Ramsey operator, which are defined using the strong subset ordering $\strsub$. Suppose $\kappa$ is a cardinal and $X$ is a set of ordinals with $\kappa\subseteq X$. For $S\subseteq P_\kappa X$, we say that $\vec{S}=\<S_x\st x\in P_\kappa X\>$ is an $(S,\strsub)$-list if $S_x\subseteq P_{\kappa_x}x$ for all $x\in S$. Given an $(S,\strsub)$-list, a set $H\subseteq S$ is said to be \emph{homogeneous for $\vec{S}$} if whenever $x,y\in H$ with $x\strsub y$ we have $S_x=S_y\cap P_{\kappa_x}x$. If $I$ is an ideal on $P_\kappa X$, we define another ideal $\I_\strsub(I)$ such that for $S\subseteq P_\kappa X$ we have $S\in \I_\strsub(I)^+$ if and only if for every $(S,\strsub)$-list $\vec{S}$ there is a set $H\in P(S)\cap I^+$ which is homogeneous for $\vec{S}$. We say that $P_\kappa X$ is \emph{strongly ineffable} if $P_\kappa X\in\I_\strsub(\NSS_{\kappa,X})^+$ and \emph{almost strongly ineffable} if $P_\kappa X\in\I_\strsub(I_{\kappa,X})^+$. Here $I_{\kappa,X}$ is the ideal on $P_\kappa X$ consisting of all subsets of $P_\kappa X$ which are not $\strsub$-cofinal in $P_\kappa X$.

Let $[S]_\strsub^{<\omega}$ be the collection of all tuples $\vec{x}=(x_0,\ldots,x_{n-1})\in S^n$ such that $n<\omega$ and $x_0\strsub\cdots\strsub x_{n-1}$. A function $f:[P_\kappa X]_\strsub^{<\omega}\to P_\kappa X$ is called \emph{$\strsub$-regressive on $S$} if $f(x_0,\ldots,x_{n-1}) \strsub x_0$ for all $(x_0,\ldots,x_{n-1})\in [S]_\strsub^{<\omega}$. Given a function $f:[P_\kappa X]_\strsub^{<\omega}\to P_\kappa X$, a set $H\subseteq P_\kappa X$ is said to be \emph{homogeneous for $f$} if $f\restrict[H]^n$ is constant for all $n<\omega$. For $S\subseteq P_\kappa X$, let $S\in\R_\strsub(I)^+$ if and only if for every function $f:[P_\kappa X]^{<\omega}\to P_\kappa X$ that is $\strsub$-regressive on $S$, there is a set $H\in P(S)\cap I^+$ which is homogeneous for $f$. We say that $P_\kappa X$ is \emph{strongly Ramsey} if $P_\kappa X\in\R_\strsub(I_{\kappa,X})^+$.

    
The first author and Philip White \cite{CodyWhite} showed that many results from the literature \cite{MR0384553, MR0540770, MR4206111, MR4594301, MR1077260} on the ineffable operator $\I$ and the Ramsey operator $\R$, and their relationship with indescribability, can be extended to $\I_\strsub$ and $\R_\strsub$. For example, by iterating the ideal operators $\I_\strsub$ and $\R_\strsub$, one obtains hierarchies in the two-cardinal setting which are analogous to the classical ineffable and Ramsey hierarchies. One question left open by \cite{CodyWhite} is that which is analogous to Question \ref{question_ramseyness} for the two-cardinal context. For example, if $P_\kappa X\in\R_\strsub^2(I_{\kappa,X})^+$, does it follow that the set
\[\{x\in P_\kappa X\st(\forall \xi<\kappa_x)\ x\in\R_\strsub(\Pi^1_\xi(\kappa_x, x))^+\}\]
is in the filter $\R_\strsub(I_{\kappa,X})^*$?


The proof of Theorem \ref{theorem_ramsey} generalizes in a straight-forward way to establish the following. 

\begin{theorem}
Suppose $\gamma<|X|^+$, $S\in\R_{\strsub}^{\gamma+1}(I_{\kappa,X})^+$ and let 
\[T=\{x\in S\st(\forall \xi<\kappa_x)\ S\cap P_{\kappa_x}x\in \R_{\strsub}^{f_\gamma(x)}(\Pi^1_\xi(\kappa_x,x))^+\}.\]
Then $S\setminus T\in\R_{\strsub}(I_{\kappa,X})$.
\end{theorem}

\begin{corollary}
Suppose $\gamma<|X|^+$. If $P_\kappa X\in\R_{\strsub}(I_{\kappa,X})^+$, then the set
\[\{x\in P_\kappa X\st (\forall\xi<{\kappa_x})\ P_{\kappa_x}x\in\R_\strsub^{f_\gamma(x)}(\Pi^1_\xi(\kappa_x,x))^+\}\]
is in the filter $\R_{\strsub}^{\gamma+1}(I_{\kappa,X})^*$.
\end{corollary}

\section{Questions and ideas}

Let us formulate a few open questions relavant the topics of this article. For this section, let us assume $\kappa$ is some regular uncountable cardinal and $X\supseteq\kappa$ is a set of ordinals. First, we consider the following questions regarding the consistency strength of various principles considered above.

\begin{question}
What is the consistency strength of ``whenever $S\subseteq P_\kappa X$ is strongly stationary there is some $x\in P_\kappa X$ for which $S\cap P_{\kappa_x}x$ is strongly stationary in $P_{\kappa_x}x$''? Is this similar to the situation for cardinals? Is the strength of this kind of reflection of strong stationary sets strictly between the ``great Mahloness'' of $P_\kappa X$ and the $\Pi^1_1$-indescribability of $P_\kappa X$?
\end{question}

\begin{question}
What is the consistency strength of the $2$-s-strong stationarity of $P_\kappa X$? What is the consistency strength of the hypothesis that whenever $S$ and $T$ are strongly stationary in $P_\kappa X$ there is some $x\in P_\kappa X$ such that $S$ and $T$ are both strongly stationary in $P_{\kappa_x}x$?
\end{question}

The following questions regarding separation of various properties considered in this article remain open.

\begin{question}\label{question_separate1}
Can we separate reflection of strongly stationary sets from pairwise simultaneous reflection of strongly stationary sets? In other words, is it consistent that whenever $S$ is strongly stationary in $P_\kappa X$ there is some $x\in P_\kappa X$ such that $S$ is strongly stationary in $P_{\kappa_x} x$, but at the same time, pairwise reflection fails in the sense that there exists a pair $S,T$ of strongly stationary subsets of $P_\kappa X$ such that for every $x\in P_\kappa X$ both $S$ and $T$ are not strongly stationary in $P_{\kappa_x}x$?
\end{question}

It is conceivable that some two-cardinal $\Box(\kappa)$-like principle could be used to address Questions \ref{question_separate1}. For example, $\Box(\kappa)$ implies that every stationary subset of $\kappa$ can be partitioned into two disjoint stationary sets that do not simultaneously reflect (see \cite[Theorem 2.1]{MR3730566} as well as \cite[Theorem 7.1]{MR4230485} and \cite[Theorem 3.50]{MR4583072} for generalizations).

\begin{question}
Is some two-cardinal $\Box(\kappa)$-like principle formulated using weak clubs (defined in Section \ref{section_strong_stationarity}) consistent? Does it deny pairwise simultaneous reflection of strongly stationary subsets of $P_\kappa X$? 
\end{question}


It is also natural to ask whether the various reflection properties introduced here can be separated from the large cardinal notions that imply them.

\begin{question}
Can we separate $\xi+1$-strong stationarity or $\xi+1$-s-strong stationarity in $P_\kappa X$ from 
\begin{enumerate}
    \item $\Pi^1_\xi$-indescribability in $P_\kappa X$ similar to what was done in \cite{MR4094556}; or 
    \item $\Pi^1_1$-indescribability in $P_\kappa X$ similar to what was done in \cite{BZ}? 
\end{enumerate}
\end{question}

In \cite{MR4094556}, it was shown that consistently $\NS^{\xi+1}_\kappa$ can be non-trivial and $\kappa$ is not $\Pi^1_{\xi}$-indescribable. In \cite[Definition 0.7]{BZ}, a normal version of the ideal $\NS_\kappa^\xi$ was introduced, $\NS^{\xi, d}_\kappa$. It was shown that consistently, $\NS^{\xi, d}_\kappa$ can be non-trivial for all $\xi<\omega$ while $\kappa$ is not even $\Pi^1_1$-indescribable.

\begin{question}
Is it consistent that $\kappa\in \I(\Pi^1_{\xi}(\kappa))$ and $\kappa\not \in \I(\NS^{\xi+1}_\kappa)$.
Is it consistent that $\kappa\in \I(\Pi^1_1(\kappa))$ and $\kappa\notin\I(\NS_\kappa^{\xi,d})$ for all $\xi<\omega$?
\end{question}

Finally, let us consider some questions that arise by considering Proposition \ref{proposition_kappa_x_plus_1} and \cite{CodyHigherIndescribability}. Bagaria noticed that, using the definitions of \cite{MR3894041}, no ordinal $\alpha$ is $\alpha+1$-stationary (see the discussion after Definition 2.6 in \cite{MR3894041}) and no cardinal $\kappa$ is $\Pi^1_\kappa$-indescribable (see the discussion after Definition 4.2 in \cite{MR3894041}). The first author showed that Bagaria's definitions of $\xi$-s-stationarity and derived topologies $\<\tau_\xi\st\xi<\delta\>$ on an ordinal $\delta$, can be modified in a natural way so that a regular cardinal $\mu$ can cary a longer sequence of derived topologies $\<\tau_\xi\st \xi<\mu^+\>$, such that, for each $\xi<\mu$ there is a club $C_\xi$ in $\delta$ such that $\alpha\in C_\xi$ is not isolated in the $\tau_\xi$ topology if and only if $\alpha$ is $f^\mu_\xi(\alpha)$-s-stationary\footnote{Here $f^\mu_\xi:\mu\to\mu$ denotes the $\xi$-th canonical function on $\mu$.} (see \cite[Theorem 6.15]{CodyHigherIndescribability}). The first author also generlized Bagaria's notion of $\Pi^1_\xi$-indescribability so that a cardinal $\kappa$ can be $\Pi^1_\xi$-indescribable for all $\xi<\kappa^+$, and that the $\Pi^1_\xi$-indescribability of $\kappa$ implies the $\xi+1$-s-stationarity of $\kappa$ for all $\xi<\kappa^+$ (see \cite[Proposition 6.18]{CodyHigherIndescribability}). It is natural to ask whether similar techniques can be used to generalize the results in Section \ref{section_derived_topologies} of the present article. For example, can one modify the definition of $\xi$-strong stationarity so that Proposition \ref{proposition_kappa_x_plus_1} can fail for the modified notion?

\begin{question}\label{question_higher1}
Can one use canonical functions to modify the definition of $\xi$-s-strong stationarity so that it is possible for $x\in P_\kappa X$ to be $\xi$-strongly stationary or $\xi$-s-strongly stationary for some $\xi>\kappa_x$?
\end{question}

\begin{question}\label{question_higher2}
Can the definitions of two-cardinal $\Pi^1_\xi$-indescribability (Definition \ref{definition_indescribability}), $\xi+1$-s-strong stationarity (Definition \ref{definition_xi_s_stationary}), and the two-cardinal derived topologies (see Section \ref{section_derived_topologies}) be modified using canonical functions so that Corollary \ref{corollary_stationarity_from_indescribability} might generalize to values of $\xi$ for which $\kappa_x<\xi<|x|^+$ and Theorem \ref{theorem_induction} might generalize to values of $\xi$ for which $\kappa<\xi<|X|^+$?
\end{question}



\end{document}